\documentclass[12pt]{amsart}

\pdfoutput=1

\usepackage{graphicx}
\usepackage{color}
\usepackage{bm}
\usepackage{subfigure}
\usepackage{mathabx}
\usepackage{multirow}
\usepackage{setspace}
\usepackage{appendix}
\usepackage{amscd,amsthm,amssymb,amsfonts,amsmath,mathrsfs,amsrefs}
\usepackage{todonotes}
\usepackage{enumitem}
\usepackage{csquotes}

\allowdisplaybreaks
\usepackage{geometry}
\newcommand{\Tr}{\operatorname{Tr}}

\newcommand{\tr}{\operatorname{tr}}
\newcommand{\e}{\epsilon}
\newcommand{\Real}{\operatorname{Re}}
\newcommand{\dsst}{\displaystyle}
\newcommand{\vol}{\mathrm{vol}}
\newcommand{\ip}[2]{\langle #1, #2 \rangle}

\newcommand{\M}{\mathcal{M}}

\newcommand{\U}{\mathcal{U}}

\newcommand{\dstar}{d^\ast}

\newcommand{\dstarh}{d^{\ast_h}}

\newcommand{\im}{\operatorname{im}}
\newcommand{\End}{\operatorname{End}}

\newcommand{\Hom}{\operatorname{Hom}}
\newcommand{\sign}{\operatorname{sign}}

\newcommand{\R}{\mathcal{R}}
\newcommand{\C}{\mathcal{C}}
\newcommand{\Mprod}{\mathcal{M}_{\text{prod}}}

\newcommand{\conj}[1]{\overline{#1}}
\newcommand{\Trconj}{\conj{\mathrm{Tr}}\,}

\newtheorem{theorem}{Theorem}[section]
\newtheorem*{theorem*}{Theorem}
\newtheorem{lemma}[theorem]{Lemma}
\newtheorem{prop}[theorem]{Proposition}
\newtheorem{coro}[theorem]{Corollary}

\theoremstyle{definition}
\newtheorem{defn}[theorem]{Definition}

\theoremstyle{remark}
\newtheorem{remark}[theorem]{Remark}
\newtheorem*{remark*}{Remark}

\numberwithin{equation}{section}

\usepackage[pdfusetitle, plainpages=false, bookmarks, bookmarksnumbered,
				colorlinks, linkcolor=black, citecolor=black,
	         filecolor=black, urlcolor=black]{hyperref}

\begin{document}			

			\title[A generalization
			of analytic torsion]{A generalization of analytic torsion
				via differential forms on spaces of metrics}
		
		\author{Phillip Andreae}
		
		\date{July 29, 2021}
		
	\address{Mathematics \& Computer Science Department,
		Meredith College, 3800 Hillsborough St.,
		Raleigh, NC 27607}
		 \email{pvandreae@meredith.edu}

															\begin{abstract}			
																We introduce multi-torsion, a spectral invariant
																generalizing Ray-Singer
																analytic torsion.
																We define multi-torsion
																for compact manifolds with a certain local
																geometric product structure
																that gives a bigrading
																on differential forms.
																We prove that multi-torsion is metric-independent
																in a suitable sense.
																Our definition of multi-torsion
																is inspired by an interpretation of each of analytic torsion
																and the eta invariant
																as a regularized integral of a closed differential form on a space of metrics on a vector bundle
																or on a space of elliptic operators.
																We generalize the Stokes' theorem
																argument explaining the dependence of torsion and eta on the geometric data used to define them
																to the local product setting to prove
																our metric-independence theorem for multi-torsion.
																		\end{abstract}
						
						\thanks{The author thanks
							Mark Stern for many useful conversations.}

						\maketitle
			


	\setcounter{tocdepth}{1}
	\tableofcontents

	
	\section{Introduction}
A natural way to obtain geometric invariants
	is to integrate closed differential forms over cycles in a space of operators or metrics,
	e.g., a space of metrics of special holonomy.
Since we will consider the space of all Riemannian metrics,
which is contractible,
	we will in fact consider cycles that are closed
	 relative to the cylindrical ends of that space.
	 After zero-forms
	(e.g., the index, a locally constant function on
	the space of elliptic operators),
	the simplest case is degree one.
	We discover two closed one-forms defined in terms of spectral data
	whose regularized integrals over a natural path
	give the analytic torsion and eta invariant, respectively.
	We use this idea to provide a new interpretation
	of the metric-invariance properties of these (relative) invariants.
	
	We next seek closed forms of higher degree.
	Here the story is still in its infancy,
	but we find a closed two-form in the case
	of a simple special holonomy: reducible holonomy.
	Our two-form, which we call $\omega_{MT}$, is the wedge product of the torsion one-forms corresponding to the
	factors in the case of a global product,
	which inspires our definition of $\omega_{MT}$ in the general case.
	Integrating $\omega_{MT}$ over a natural two-cycle gives
	a new invariant that we call multi-torsion.
	We investigate its properties
	and prove a metric independence theorem under appropriate conditions.
	In future work, we plan to study the application of multi-torsion
	to the most interesting cases of reducible holonomy
	for number theory: Hilbert modular varieties and their relatives.

		\subsection{Motivation and background: Analytic torsion}
		Ray-Singer  \cites{ray-singer-71, ray-singer-complex} introduced the analytic torsion $T$
		in the early 1970's	 as a geometric-analytic analogue
		of the Reidemeister torsion $\tau$,
		a classical combinatorial invariant.
		Ray-Singer's conjecture that $T = \tau$
		was proven independently
		by Cheeger \cite{cheeger} and M\"{u}ller \cite{muller-78} and later generalized by M\"{u}ller \cite{muller-93} and Bismut-Zhang \cite{bismut-zhang}, among others.	
		Among the consequences of the  Cheeger-M\"{u}ller theorem
		is the surprising fact that 
		torsion in integral cohomology can be studied with analytic methods.
		This fact  is central to some recent work in
		number theory and arithmetic geometry,
		much of it centered around a conjecture
		of Bergeron-Venkatesh about the growth of torsion in the homology
		of arithmetic groups \cite{bergeron-venkatesh}.
		There have also been several recent generalizations of analytic torsion
		to non-compact and singular manifolds
		(e.g., \cites{albin-rochon-sher-2018-wedge,albin-rochon-sher-2018-witt,lesch, mazzeo-vertman,muller-pfaff-1, muller-pfaff-2, muller-pfaff-3, vertman}).
		
		Our multi-torsion is intended to be a novel
		``higher" torsion,
		several notions of which already exist in the literature,
		including the real analytic torsion forms of Bismut-Lott \cite{bismut-lott}.
		(See also the holomorphic torsion forms of
		Bismut-Gillet-Soul\'{e} \cites{bgsI, bgsII, bgsIII}.)
		For nonnegative even integers $j$,
		Bismut-Lott construct higher analytic torsion forms $\mathscr{T}_{j}$
		that are differential forms of degree $j$ on the base of a fiber bundle;
		$\mathscr{T}_0$ is 
		half the de Rham analytic torsion of the fibers.
		To our knowledge, there is not a direct connection between 
		our forms and the forms $\mathscr{T}_j$.
		Whereas the degree-zero part of the Bismut-Lott higher torsion forms
		is the analytic torsion itself,
		we \emph{integrate} our forms to produce analytic torsion
		and multi-torsion, respectively.
		
		To review the de Rham torsion
		and to introduce this paper's approach, let us
		consider a compact smooth $n$-manifold $M$
		and a flat vector bundle $F \to M$.
		Let $\Lambda^q F$ denote the bundle of $F$-valued
		$q$-forms.
		The analytic torsion $T = T(M, \Lambda^\bullet F)$ 
		is defined by
		\begin{align} \label{eq: torsion in terms of det}
		\log T = - \frac{1}{2} \sum_{q=0}^n (-1)^q q \log \det \Delta_q,
		\end{align}
		which may be interpreted as the zeta-regularization of the divergent integral
		\begin{equation} \label{eq: integral defining torsion}
		\frac{1}{2} \int_0^\infty \sum_{q=0}^n (-1)^q q \Tr' e^{-t \Delta_q} \frac{dt}{t},
		\end{equation}
		where $\Delta_q$ denotes the Laplacian
		on $\Lambda^q F$ and $\Tr'$ denotes
		the trace on $(\ker \Delta)^\perp$.
	 	Ray-Singer proved that $T$ is independent of the choice of Riemannian metric
		assuming that $F$ is acyclic \cite{ray-singer-71}.
		We provide a new interpretation of that result
		by introducing the following one-form $\omega_{T}$
		on the space of metrics, which we prove is closed:
		\begin{align} \label{eq: omega RS}
		\omega_{T} := \sum_{q=0}^n (-1)^q  \Tr' e^{-t \Delta_q} (h^q)^{-1} (\delta h^q).
		\end{align}
		The regularized integral of $\omega_T$
		along a radial curve in the space of metrics
		gives $2 \log T$.
		In \eqref{eq: omega RS}, $h^q$ denotes the metric on $q$-forms,
		viewed as a bundle isomorphism $\Lambda^q F \to (\Lambda^q F)^\ast$,
		and $\delta$ denotes
		the exterior derivative on differential forms on the space of such metrics.
		Note that $\omega_T$ generalizes
		the one-form $\alpha$
		of Bismut-Zhang \cite{bismut-zhang} (see Remark \ref{remark: bismut-zhang}).

		We explain in the sequel how this approach
		generalizes to the setting
		of a $\mathbb{Z}$-graded elliptic complex,
		introduced by Schwarz \cite{schwarz},
		which is a special case of Mathai-Wu's analytic torsion
		for $\mathbb{Z}_2$-graded elliptic complexes \cite{mathai-wu}.
		This then provides a model for our approach to multi-torsion,
		for which a bigrading plays an essential role.

		\subsection{Definition of multi-torsion and main results}
		Suppose that $M_1$ and $M_2$
		are compact manifolds and  that $F_1 \to M_1$ and $F_2 \to M_2$ are orthogonal
		or unitary flat vector bundles.
		We must assume that $F_1$ and $F_2$ are both acyclic.
		Let $F = \pi_1^\ast F_1 \otimes \pi_2^\ast F_2 \to M_1 \times M_2$
		be the flat product vector bundle,
		where for $j=1,2$, $\pi_j: M_1 \times M_2 \to M_j$ is the projection.
		Let $\Gamma \subset \operatorname{Diff}(M_1) \times \operatorname{Diff}(M_2) \subset
		\operatorname{Diff}(M_1 \times M_2)$ be a group of diffeomorphisms
		that lift to vector bundle isomorphisms
		preserving the flat structures on $F_1$ and $F_2$.
		(See \S \ref{subsection: The geometry of finite quotients of product manifolds}
		for the details.)
		For $j=1,2$, let $\Gamma_j$ denote the projection of $\Gamma$
		onto $\operatorname{Diff}(M_j)$.
		We assume that $\Gamma$ is finite, although we expect that this assumption could be weakened.
		
		There is  a quotient flat bundle $F_\Gamma \to M_\Gamma$
		over the quotient $M_\Gamma := \Gamma \backslash (M_1 \times M_2)$,
		which inherits a local product structure,
		giving a bigrading of $F_\Gamma$-valued forms
		into ``$(q_1, q_2)$-forms"
		(for $0 \leq q_j \leq n_j$)
		and a decomposition of $d$:
		$d = d_1 + d_2$,
		where $d_1$ maps $(q_1, q_2)$-forms to $(q_1 + 1, q_2)$-forms
		and $d_2$ is similar.
		If in addition we endow $M_1 \times M_2$ with a $\Gamma$-invariant product Riemannian metric
		$h_1 \times h_2$,
		then we obtain a decomposition of the Laplacian on $F_\Gamma$-valued forms:
		$\Delta = \Delta_1 + \Delta_2$, where for $j=1,2$,
		$\Delta_j := d_j d^\ast_j + d^\ast_j d_j$.
		For $t_1, t_2 >0$, let
		$\Delta(t_1, t_2) := t_1 \Delta_1 + t_2 \Delta_2$;
		write $\Delta^{q_1, q_2}(t_1, t_2)$
		for its restriction to $(q_1, q_2)$-forms.
		We define the multi-torsion $MT = MT(M_\Gamma, F_\Gamma, h_1, h_2)$ by
		\begin{align*}
		MT& := \frac{1}{4}	 \left. \frac{\partial^2}{\partial s_1 \partial s_2} \right|_{(s_1, s_2)=(0,0)} \zeta(s_1, s_2), \text{ where for $\Real s_1, \Real s_2$ large,} \\
		\zeta(s_1, s_2) &:=	\frac{1}{\Gamma(s_1) \Gamma(s_2)}  
		\int_0^\infty \int_0^\infty 
		\sum_{q_1, q_2} (-1)^{q_1 + q_2} q_1 q_2 \Tr e^{-\Delta^{q_1, q_2}(t_1, t_2)}  t_1^{s_1} t_2^{s_2}  \, \frac{dt_1}{t_1} \frac{dt_2}{t_2}.
		\end{align*}
		We view $MT$ as a regularization of the
		following divergent  integral,
		which is a two-variable
		generalization of \eqref{eq: integral defining torsion}:
		\begin{equation*}
		\frac{1}{4} \int_0^\infty \int_0^\infty 
		\sum_{q_1, q_2} (-1)^{q_1 + q_2} q_1 q_2 \Tr e^{-\Delta^{q_1, q_2}(t_1, t_2)} \, \frac{dt_1}{t_1} \frac{dt_2}{t_2}.
		\end{equation*}
		To ensure that $MT$ is well-defined,
		we study the asymptotic properties of the two-parameter heat kernel $e^{-\Delta(t_1, t_2)}$
		and develop a theory of ``multi-zeta" functions
		such as $\zeta(s_1, s_2)$ above.
		These are functions of two complex variables resembling the zeta functions of Shintani \cite{shintani}.

		It is apparent that in the case when $\Gamma$ is the trivial group,
		we have the product formula
		\begin{align*}
		MT(M_1 \times M_2, \pi_1^\ast F_1 \otimes \pi_2^\ast F_2, h_1, h_2  ) 
		&= \log T(M_1, \Lambda^\bullet F_1, h_1) \log T(M_2,  \Lambda^\bullet F_2, h_2),
		\end{align*}
		where $T$ is the Ray-Singer analytic torsion.
		This motivates the general case,
		but in general, $MT$ 
		does not decompose as a product.

		Our main results concern the dependence of $MT$
		on the local product metric.
		First, we obtain the following
		generalization of the vanishing
		of the de Rham torsion in even dimensions \cite{ray-singer-71}.
		Our proof uses the Hodge star operator.
		\begin{theorem*}[Theorem \ref{theorem: even vanishing of MT}]
			Suppose that for at least one of either   $j=1$ or $j=2$,
			$n_j$ is even and $\Gamma_j$ acts on $M_j$ by orientation-preserving diffeomorphisms.
			Then $MT(h_1, h_2) = 0$ for any metrics $h_1$ and $h_2$.
		\end{theorem*}
		Our main theorem  is the following metric independence
		theorem in odd dimensions.
		\begin{theorem*}[Theorem \ref{theorem: multi torsion metric independence}]
			Suppose that for either  $j = 1$ or $j=2$,
			$\dim M_j$ is odd and for every $\gamma_j \in \Gamma_j $,
			$\gamma_j$ is either orientation-preserving or has nondegenerate fixed points as a diffeomorphism of $M_j$.
			Then $MT(M_\Gamma, F_\Gamma, h_1, h_2)$ is independent of the metric $h_j$.
		\end{theorem*}
		
		Our proof 
		hinges on the central principle of this paper:
	 	for a certain closed two-form $\omega_{MT}$
		on the space of local product metrics,
		$MT$ is a regularized integral
		of $\omega_{MT}$ over the surface parametrized by
		$t_1$ and $t_2$.
		The variation of $MT$
		with respect to a change in metric therefore
		reduces to locally computable boundary terms that vanish under the assumptions
		of Theorem \ref{theorem: multi torsion metric independence}.

		\subsection{Outline.} This paper is organized as follows.
	In \S \ref{section: Preliminaries},  we collect some prerequisite results
	on differential forms on spaces of metrics and on heat kernels
	and zeta functions.
	In \S \ref{section: Analytic torsion of an elliptic complex},
	we review the analytic torsion of an elliptic complex
	and set up our differential form perspective
	in this simple model case.
	In \S \ref{section: multi-torsion},
	we define
	multi-zeta functions and multi-torsion.
	In \S \ref{sec: Heat kernel multi-asymptotics},
	we prove some ``multi-asymptotic"
	properties of two-parameter heat kernels.
	In \S \ref{sec: Metric independence theorem for multi-torsion},
	we introduce the multi-torsion form $\omega_{MT}$, prove that it is closed, and use this fact to prove our main theorem
	 on the metric independence of multi-torsion.
	In \S \ref{section: eta},
	which is independent of \S \ref{section: Analytic torsion of an elliptic complex}-\ref{sec: Metric independence theorem for multi-torsion},
	we study the metric dependence of
	the eta invariant
	via a closed differential form
	on the space of elliptic operators.


\section{Preliminaries}

\label{section: Preliminaries}

\subsection{Operator-valued differential forms }

\label{subsection: Differential forms}

Let $\mathcal{U}$ be a smooth manifold.
We will assume that $\U$ is finite-dimensional,
which is sufficient to obtain our results.
But in applications,
we will think of $\mathcal{U}$ as a space
parametrizing a family of metrics or a family of differential operators,
and so it will be useful to view
$\mathcal{U}$ as a submanifold
of an infinite-dimensional manifold
(such as the space of all metrics on some vector bundle).

For an associative $\mathbb{R}$- or $\mathbb{C}$-algebra $A$,
let $\Omega^k(\U, A)$ denote the space of smooth $A$-valued
exterior differential forms of degree $k$ on $\U$;
let $\Omega(\U, A) := \bigoplus_{k=0}^{\dim \U} \Omega^k(\U, A)$
denote the space of smooth $A$-valued
exterior differential forms (of all degrees) on $\U$.
For $u \in \U$ and $X_1, \dots X_k$ tangent vectors
to $\U$ at $u$,
we will write
$\omega(u; X_1, \dots, X_k) \in A$
for the evaluation of $\omega \in \Omega^k(\U, A)$
at $u$ against the $k$-tuple $(X_1, \dots, X_k)$;
by definition, for each fixed $u$,
$\omega(u; \cdot, \cdots, \cdot )$
is multilinear and alternating.

The product structure on $A$ induces
a wedge product structure on $\Omega(\U, A)$,
which need not be graded commutative in general
(but see Lemma \ref{lemma: adjoint}).
For simplicity of notation, we will omit the traditional symbol ``$\wedge$'';
i.e., for $\kappa \in \Omega^k(\U, A)$ and $\lambda \in \Omega^l(\U, A) $,
their wedge product will be denoted simply by $\kappa \lambda \in \Omega^{k+l}(\M, A)$.

We will denote the exterior derivative $\Omega^k(\U, A) \to \Omega^{k+1}(\U, A) $
by $\delta^\mathcal{U}$, or simply by $\delta$ if $\U$ is understood.
If $\delta^\U \omega = 0$, then we will say $\omega$ is closed.	
The following familiar properties
hold in this context: $\delta^\U \circ \delta^\U = 0$,
the graded Leibniz rule, and Stokes' theorem.

Now let $M$ be a compact manifold, and let $E \to M$ be a real or complex vector bundle.
By a metric on $E$, we shall mean a smoothly varying inner product on the fibers of $E$;
this inner product is euclidean (if $E$ is a real vector bundle) or hermitian (if $E$ is a complex vector bundle).
To simplify the exposition, we will assume
that $E$ is a real vector bundle for now,
but the results hold in the complex case with minor modifications.
Let $\M(E)$, or simply $\M$ if $E$ is understood,
denote the set of all metrics on $E$.
which is a Fr\'{e}chet manifold.

Let us choose a density $\vol_0$ on the base manifold $M$.
Then each choice of metric $h \in \M$ induces an $L^2$-inner product on smooth sections of $E$, $\ip{\cdot}{\cdot}_h$, defined by
\begin{equation} \label{equation: inner product}
\ip{\cdot}{\cdot}_h : = \int_M h(\cdot, \cdot)_x \, \vol_0(x),
\end{equation}
where $h(\cdot, \cdot)_x$ denotes the inner product on the fiber $E_x$.
The $L^2$ space $L^2_h(M, E)$ and the $L^2$-Sobolev spaces $H^{s}_h(M, E)$ for $s \in \mathbb{R}$, are defined in the usual way.
Any two metrics
induce equivalent norms since $M$ is compact, so the topological vector spaces
$L^2(M, E)$ and $H^{s}(M, E)$,
and the notion of a bounded operator, are independent of the metric.
Let $\mathcal{B}$ denote the unital associative algebra of bounded
linear endomorphisms of $L^2(M, E)$.

\begin{remark} \label{remark: fix density}
	In \eqref{equation: inner product}, replacing $\vol_0$ by $\varphi \, \vol_0$,
	where $\varphi: M \to \mathbb{R}$ is a smooth positive function,
	induces the same $L^2$ inner product as replacing $h$ by $\varphi h$ and keeping $\vol_0$ the same.
	Because of this redundancy, we will sometimes decide to fix a density $\vol_0$.
\end{remark}

Let $\Psi$ denote the unital associative algebra of pseudodifferential operators
$C^\infty(M, E) \to C^\infty(M, E)$,
and for $s \in \mathbb{R}$, let $\Psi^s$ denote the subset of $\Psi$ consisting of pseudodifferential operators
of order (less than or equal to) $s$.
Note that if $s < t$, then $\Psi^s \subset \Psi^t$.
If $s \leq 0$, then $\Psi^s$ forms an algebra since it is closed under composition
(in general, $\Psi^s \cdot \Psi^t \subset \Psi^{s+t}$).
Note that neither $\mathcal{B}$ nor $\Psi$ is contained in the other,
but if $s \leq 0$, then $\Psi^s \subset \mathcal{B}$.

For a metric $h \in \M$ and an operator $\phi$ either in $\mathcal{B}$ or in $\Psi$,
the $h$-adjoint of $\phi$, denoted $\phi^{\ast_h}$,
is the operator either in $\mathcal{B}$ or in $\Psi$, respectively,
characterized by, for smooth sections $a,b$ of $E$,
\begin{equation*}
\ip{\phi a}{b}_h = \ip{a}{\phi^{\ast_h} b}_h.
\end{equation*}
If $\phi \in \mathcal{B}$, then
since smooth sections are dense in $L^2$ sections, this defines $\phi^{\ast_h}$ uniquely
as an element of $\mathcal{B}$.
We will say that an operator $\phi$
is symmetric with respect to $h$ if $\phi^{\ast_h} = \phi$
(on smooth sections).

Now suppose that we have fixed a map $\U \to \M$;
$\U$ may be a submanifold of $\M$, in which case we
take this map to be the inclusion $\U \hookrightarrow \M$.
This allows us to define the adjoint $\omega^\ast \in \Omega^k(\U, A)$
of a differential form $\omega \in \Omega^k(\U, A)$,
where $A = \mathcal{B}$ or $A = \Psi$,
by
\begin{equation*}
(\omega^\ast)(u; X_1, \dots, X_k) := (\omega(u;X_1, \dots, X_k))^{\ast_u}
\end{equation*}
for all $u \in \U$ and for all $X_1, \dots, X_k$ that are tangent vectors
to $\U$ at $u$.
The latter adjoint is defined with respect
to the metric associated to $u\in \U$ via the map $\U \to \M$.
We will say that $\omega$ is symmetric when $\omega^\ast = \omega$.

For a trace-class operator $\phi \in \mathcal{B}$,
its trace $\Tr \phi$ is well-defined
independent of the metric by Lidskii's theorem.
If $\phi$ is a finite-rank operator, then
we will denote its trace by $\tr \phi$.

\begin{remark} \label{remark: trace-class inclusion}
	We will find the following well-known fact useful.
	Let $n$ be the dimension of the base manifold $M$.
	Then if $s < -n$ and if $\phi \in \Psi^s \subset \mathcal{B}$, then $\phi$ is trace-class. 
\end{remark}

We will say that $\omega \in \Omega^k(\U, \mathcal{B})$ is trace-class when
for every $u \in \U$ and $X_1, \dots, X_k$ tangent to $\U$ at $u$,
$\omega(u; X_1, \dots, X_k) \in \mathcal{B}$ is trace-class.
If $\omega$ is trace-class, $\Tr \omega \in \Omega^k(\U, \mathbb{C})$ will denote the $\mathbb{C}$-valued $k$-form on $\U$
defined by
\begin{align*}
(\Tr \omega)(u; X_1, \dots, X_k) := \Tr \left(\omega(u; X_1, \dots, X_k)\right),
\end{align*}
and
$\Trconj \omega \in \Omega^k(\M, \mathbb{C})$ will denote the complex conjugate of $\Tr \omega$,
defined by
\begin{align*}
(\Trconj \omega)(u; X_1, \dots, X_k) := \conj{ \Tr \left(\omega(u; X_1, \dots, X_k)\right)}.
\end{align*}

We have the following results concerning the interactions
between products, adjoints, and traces:
\begin{lemma} \label{lemma: adjoint} 
	Let $\kappa \in \Omega^k(\U, A)$ and $\lambda \in \Omega^l(\U, A)$,
	where $A = \mathcal{B}$ or $A = \Psi$.
	Fix a map $\U \to \M$
	that will define the adjoint operation.
	Then we have:
	\begin{enumerate}
		\item $(\kappa \lambda)^\ast = (-1)^{kl} \lambda^\ast \kappa^\ast $.
		\item  $\kappa$ is trace-class if and only if $\kappa^\ast$ is trace-class, and 
		in that case, 
		$  \Tr \kappa =  \Trconj \kappa^\ast$.
		\item  If $\kappa \lambda$ is trace-class, then
		$   \Tr \kappa \lambda =  \Trconj (\kappa \lambda)^\ast  
		= (-1)^{kl} \,  \Trconj \lambda^\ast \kappa^\ast $.
		\item  If $\kappa \lambda$ and $\lambda \kappa$ are both trace-class, then
		$     \Tr \kappa \lambda = (-1)^{kl} \Tr \lambda \kappa$.
	\end{enumerate}
\end{lemma}

\begin{proof}
	The properties follow from the corresponding properties for operators
	and the alternating property of differential forms.
\end{proof}

\begin{coro}  
	Let $\kappa_1, \dots, \kappa_r \in \Omega^1(\U, A)$,
	where $A = \mathcal{B}$ or $A = \Psi$.
	Fix a map $\U \to \M$
	that will define the adjoint operation.
	Let $\epsilon_r := (-1)^{\frac{1}{2}r(r-1)}$.
	Then 
	\begin{equation*}
	(\kappa_1 \dotsm \kappa_r)^\ast = \epsilon_r \, \kappa_r^\ast \dotsm \kappa_1^\ast,
	\end{equation*}
	and if $\kappa_1 \dotsm \kappa_r$ is trace-class, then
	\begin{align*}
	\Tr \kappa_1 \dotsm \kappa_r &=  \Trconj (\kappa_1 \dotsm \kappa_r)^\ast  \\
	&= \epsilon_r  \Trconj \kappa_r^\ast \dotsm \kappa_1^\ast .
	\end{align*}
\end{coro}
\begin{proof}
	The assertions follow from Lemma \ref{lemma: adjoint} by induction on $r$.
\end{proof}

Finally,  the trace commutes with the exterior derivative:
\begin{lemma}
	If $\kappa \in \Omega^k(\U, \mathcal{B})$ is trace-class, then $\delta^\U \kappa$ is also trace-class,
	and $\delta^\U (\Tr \kappa) = \Tr (\delta^\U \kappa)$.
\end{lemma}

\subsection{Heat kernels and zeta functions}
\label{subsection: Heat kernels and zeta functions}

Here we establish notation and recall some
well-known results,
originally due to, among others,
Minakshisundaram-Pleijel \cite{minakshisundaram-pleijel} and
Seeley \cite{seeley}, regarding elliptic operators
and their associated heat kernels and zeta functions.

\begin{defn} \label{defn: admissible}
	We will say that a smooth $\mathbb{C}$-valued function $h(t)$, $t > 0$,
	is admissible when it satisfies both of the following conditions:
	\begin{enumerate}[label={(A\arabic*)}]
		\item \label{A1}  There exist real powers $p_0 < p_1 < p_2 < \cdots \to \infty$
		and complex coefficients $a_{p_0}, a_{p_1}, a_{p_2}, \dots $
		such that for every real number $K$, there exists $J_K$ such that
		\begin{equation}  \label{eqn: abstract as exp}
		h(t) = \sum_{j=0}^{J_K}  a_{p_j} t^{p_j} + r_K(t),
		\end{equation}
		where the remainder $r_K(t)$ is $O(t^K)$ as $t \to 0^+$.
		In this case, we will write $h(t) \sim \sum_{j \geq 0} a_{p_j} t^{p_j}$
		and refer to this as the asymptotic expansion of $h(t)$ 
		as $t \to 0^+$.
		For each $q$ such that $t^q$ does not appear in this asymptotic expansion, our convention is to set $a_q = 0$.
		We will use the notation
		$[h(t)]_{t^q} := a_q$.
		\item \label{A2} There exists $\epsilon >0$
		such that $h(t)$ is $O(e^{-\epsilon t})$ as $t \to \infty$.
	\end{enumerate}
\end{defn}
The Mellin transform of an admissible function $h(t)$
gives its associated zeta function $\zeta(s)$,
defined by the following, for $s \in \mathbb{C}$ with
$\Real s $ sufficiently large:
\begin{equation} \label{equation: def of zeta}
\zeta(s) :=  \frac{1}{\Gamma(s)} \int_0^\infty   t^s \, h(t) \, \frac{dt}{t}.
\end{equation}
One can compute the integral in
\eqref{equation: def of zeta} explicitly
up to a holomorphic remainder
to obtain:
\begin{lemma} \label{lemma: zeta properties}
	Suppose $h(t)$ is admissible.
	Let $I(s) := \int_0^\infty   t^s \, h(t) \, \frac{dt}{t}$.
	Then the integral defining $I(s)$ converges for $\Real s > -p_0$
	and defines a holomorphic function there. Furthermore,
	$I(s)$ admits a unique meromorphic extension to all of $\mathbb{C}$
	whose only poles are simple poles at (for $j = 0, 1, 2, \dots$)
	$s = - p_j$, at which the residue is $[h(t)]_{t^{p_j}}$.
\end{lemma}

From the facts that $\frac{1}{\Gamma(s)}$ is an entire function
and $\frac{1}{\Gamma(s)} = s + O(s^2)$ as $s \to 0$, one obtains:

\begin{coro} \label{coro: zeta properties}
	Suppose $h(t)$ is admissible. Then the associated zeta function
	$\zeta(s)$ admits a unique meromorphic extension to
	all of $\mathbb{C}$. Furthermore, the extension, which we also denote by $\zeta(s)$, is holomorphic at the origin, where its value is
	$  \zeta(0) = [h(t)]_{t^0}$.
	The extension has at worst a simple pole at $s=\frac{1}{2}$, at which the
	residue is $\frac{1}{\Gamma(\frac{1}{2})}[h(t)]_{t^{-1/2}}$.
\end{coro}

%

Now let $L$ be an elliptic differential operator of order $m$
on sections of a vector bundle $E$ over a compact manifold $M$
of dimension $n$.
Assume $L$ is nonnegative and symmetric
with respect to some  metric on $E$.
Let $e^{-tL}$ denote the heat operator.
Let $Q$ be a differential operator of order $q \geq 0$ on
sections of $E$.
Let $\Pi_{\ker L}$ denote the orthogonal projection
$L^2(M, E) \to \ker L $.
For a trace-class operator $A$, let
$\Tr' A : = \Tr A  \left( I -  \Pi_{\ker L} \right)$.
Then we have the well-known result:

\begin{theorem}
	Let $L$ and $Q$ satisfy the assumptions above.
	Then $\Tr' Q e^{-tL}$ is an admissible function of $t > 0$,
	and the asymptotic expansion of \ref{A1} takes the form
	\begin{equation}  \label{equation: heat trace-prime asymptotics}
	\Tr' Q e^{-tL} \sim - \tr Q \Pi_{\ker L} +  \sum_{j \geq 0, \, q-j \text{ even}} [\Tr Q e^{-tL}]_{t^{\frac{-n -q+ j}{m}}} t^{\frac{-n -q + j}{m}},
	\end{equation}
	where the coefficients $[\Tr Q e^{-tL}]_{t^{a}}$ are the integrals of data
	that are locally computable
	in terms of the symbols of $Q$ and $L$. 
\end{theorem}

This allows one to
define the zeta function associated to the elliptic operator $L$ and 
auxiliary operator $Q$
as in \eqref{equation: def of zeta} with $h(t) = \Tr' Q e^{-tL} $, i.e.,
\begin{equation} \label{equation: heat zeta def}
\zeta(s; L, Q) : = \frac{1}{\Gamma(s)} \int_0^\infty   t^s \, \Tr' Q e^{-tL} \, \frac{dt}{t}.
\end{equation}
In the case when $Q$ is the identity, set $\zeta(s; L ):= \zeta(s; L, I)$.
Following Ray-Singer \cite{ray-singer-71},
one defines the zeta-regularized determinant
$\det L$ by
\begin{align}
\label{equation: determinant}
\det L : = \exp \left( - \left. \frac{d}{ds}\right|_{s=0} \zeta(s; L) \right),
\end{align}
which may be interpreted as a regularized product of the nonzero eigenvalues of
$L$.

Corollary \ref{coro: zeta properties} implies immediately
the following,
which we record
for applications to analytic torsion and the eta invariant.
\begin{theorem}  \label{theorem: L zeta properties}
	The  integral defining $\zeta(s; L, Q)$ in \eqref{equation: heat zeta def} converges for $\Real s > \frac{n+q}{m}$
	and defines a holomorphic function there. Furthermore,
	$\zeta(s; L, Q)$ admits a unique meromorphic extension to all of $\mathbb{C}$.
	The extension, which we also denote by $\zeta(s; L, Q)$,
	is holomorphic at the origin, where its value is
	\begin{equation} \label{equation: value of zeta at zero}
	\zeta(0; L, Q)  = [\Tr Q e^{-tL}]_{t^0} - \tr Q \Pi_{\ker L}.
	\end{equation}
	The extension has at worst a simple pole at $s=\frac{1}{2}$, 
	where the
	residue is $\frac{1}{\Gamma(\frac{1}{2})}[\Tr Q e^{-tL}]_{t^{-1/2}}$.
\end{theorem}
\begin{coro} \label{coro: a'_0 vanishes}
	If the dimension $n$ is odd and $\ker L$ is trivial, then for any $Q$,
	$\zeta(0; L, Q)  = 0$.
\end{coro}

Finally, we note that we will often find it useful to work with
the resolvent rather than directly with the heat operator.
For $z \in \mathbb{C}$ not in the spectrum of $L$,
let $R_z := (z - L)^{-1} \in \Psi^{-m}$ denote the resolvent.
Let $N$ be an integer greater than $ \frac{n+q}{m}$, which ensures
$Q R_z^N \in \Psi^{q-Nm}$ is trace-class by Remark \ref{remark: trace-class inclusion}.
Let $C$ be a contour in the complex plane
surrounding the interval $[0, \infty) \subset \mathbb{R} \subset \mathbb{C}$.
Then the Cauchy integral formula gives that
\begin{equation} \label{eq: CIF with trace}
\Tr Q e^{-tL}   =  \frac{1}{2\pi i} \frac{1}{(N-1)!}  \int_C e^{-tz} \, \Tr Q R_z^N \, dz.
\end{equation}


\section{Analytic torsion of an elliptic complex}

\label{section: Analytic torsion of an elliptic complex}

\subsection{Definition of torsion}

\label{subsection: Definition of torsion}

Suppose $M$ is a compact manifold,
$E = \bigoplus_{q=0}^r E^q$ is a $\mathbb{Z}$-graded vector bundle over $M$,
and $d = \bigoplus_{q=0}^r d_q$ is a differential operator
giving an elliptic complex
\begin{equation*}
0 \longrightarrow C^\infty(M,E^0) \overset{d_0}{\longrightarrow} C^\infty(M,E^1) \overset{d_1}{\longrightarrow}
\dotsm \overset{d_{r-1}}{\longrightarrow} C^\infty(M,E^r) \longrightarrow 0.
\end{equation*}
In this section, let $\M$ denote the space of metrics on $E$ such that the subbundles $E^q$
are mutually orthogonal.
Fix a density $\vol_0$ on $M$,
which we may do without loss of generality
(see Remark \ref{remark: fix density}).
For each metric $h \in \M$,
let $\dstarh$ be the formal adjoint of $d$ with respect to the
$L^2$ inner product associated to $h$.
The Laplacian $\Delta^h$ associated to a metric $h \in \M$
is $\Delta^h := (d + d^{\ast_h})^2 
= d d^{\ast_h} + d^{\ast_h} d$.
Let $\Delta^h_q = d_{q-1} d^{\ast_h}_{q-1} + d^{\ast_h}_q d_q$ denote the restriction of $\Delta^h$ to sections of $E^q$.
Let $m$ be the order of $\Delta$ as a differential operator, which is twice the order of $d$.
$\Delta^h$ is elliptic and, with respect to $\ip{\cdot}{\cdot}_h$, symmetric and nonnegative.
The elliptic complex $(E,d)$ is said to be acyclic
when the cohomology
$H^q(E,d):=\ker d_q / \im d_{q-1} $ is trivial for $q=0, \dots, r$, 
or, equivalently by the Hodge theorem,
when $\Delta^h$ has trivial kernel.

The following notation will be useful:
Let $Q$ act on sections of $E^q$ as multiplication by $q$.
More generally, for $f: \mathbb{Z}_{\geq 0} \to \mathbb{C}$,
let $f(Q)$ act on sections of $E^q$ as multiplication by $f(q)$.

\begin{defn}[\cites{ray-singer-71, schwarz}] \label{definition: torsion}
	 The analytic torsion $T(M, E, h)$ associated to the elliptic complex $(E,d)$ and a metric $h \in \M$ is defined by
	\begin{equation} \label{eq: def of torsion}
	\log T(M,E, h) := \frac{1}{2}  \left. \frac{d}{ds} \right|_{s=0}  \zeta \left( s ; \Delta^h, (-1)^Q Q \right).
	\end{equation}
\end{defn}

\begin{remark} From \eqref{equation: heat zeta def},
	 \eqref{equation: determinant}, and \eqref{eq: def of torsion}
	we see that
	\begin{equation} \label{equation: rewrite torsion det Delta}
	\log T(M, E, h) = -\frac{1}{2} \sum_{q=0}^r (-1)^q q \log \det \Delta_q,
	\end{equation}
	generalizing \eqref{eq: torsion in terms of det}.
\end{remark}

\subsection{A closed form for torsion}
\label{subsection: the one-form}

Before defining $\omega_T$, we will introduce
$b$, a one-form on $\M$ with values in $C^\infty(\End E) \subset \mathcal{B}(E)$.
We will view each metric $h \in \M$
as a vector bundle isomorphism $h: E \to E^\ast$,
where $v \in E_x \mapsto h(v, \cdot)_x \in E_x^\ast$.
Let $h^{-1}$ denote the inverse isomorphism $E^\ast \to E$.
We may then view $h$ as a $C^\infty(M, \Hom(E, E^\ast))$-valued $0$-form on $\M$
and $h^{-1}$ as a $C^\infty(M, \Hom(E^\ast, E))$-valued
$0$-form on $\M$.
Let $b$ denote the following ``logarithmic derivative"
of the metric:
\begin{equation*}
b := h^{-1} \delta^{\M} h \in \Omega^1(\M, \mathcal{B}(E)).
\end{equation*}
The following lemmas give some
useful properties of $b$.
We will omit their proofs, which follow from short computations.

\begin{lemma} \label{lemma: derivative of d star and Delta}
	We have the identity
	$	\delta^\M d^\ast = [d^\ast, b]$,
	which immediately implies
	$		\delta^\M \Delta = \{ d, [d^\ast, b] \}$.
\end{lemma}

\begin{lemma} \label{lemma: b is symmetric}
	The one-form $b$ is symmetric, i.e., $b^\ast = b$.
\end{lemma}

\begin{lemma} \label{lemma: derivative of b}
	The derivative of $b$ is $\delta^{\M} b = - b \, b$.
\end{lemma}

Now let $\omega_T$ be the following $\mathbb{C}$-valued
one-form on $\M$:
\begin{align} \label{eq: def of omega T}
\omega_T := \Tr' (-1)^Q e^{-\Delta} b \in \Omega^1(\M, \mathbb{C}).
\end{align}
Note that, pulled back to a curve in $\M$ parametrized by $h(u)$ for $u$ in some interval in $\mathbb{R}$, we have $b = h(u)^{-1} \frac{d h}{d u} \, du$ and
\begin{align*}
\omega_T &= \Tr' (-1)^Q e^{-\Delta^{h(u)}} h(u)^{-1} \frac{d h}{d u} \, du.
\end{align*}

We will now explain the significance of $\omega_T$
to the analytic torsion.
In what follows, we will use the
notation $h^p$ for the restriction of a metric $h \in \M$
to the subbundle $E^p$.
\begin{remark} \label{remark: scale Riemannian metric}
	Suppose  $E = \Lambda^\bullet F$ is the bundle of $F$-valued forms,
	where $F$ is an orthogonal or unitary flat bundle.
	Suppose that $h$ is a metric on $E$ that
	is induced by a Riemannian metric $g^{TM}$ and the canonical metric on $F$.
	For $t>0$, let $g^{TM}_t := \frac{1}{t} g^{TM}$, and let $h_t$ be the induced metric on $E$.
	Then it follows that $h_t^p = t^p h^p$.
\end{remark}
This special case motivates the following more general construction in the presence of a $\mathbb{Z}$-grading.
For some fixed metric $h \in \M$ and for $t > 0$, let $h_t = t^Q h$,
which is the metric on $E$
whose restriction to $E^q$ is $t^{q} h^q$.
This parametrizes a curve $C_h$ in $\M$.
Pulled back to $C_h$, we have 
that $\delta^\M h_t = Q t^{Q-1} h \, dt =  Q h \frac{dt}{t}$
and therefore that the one-form $b = h_t^{-1} \delta^\M h_t $ is
$  b = Q \frac{dt}{t}$.
From Lemma \ref{lemma: derivative of d star and Delta},
we see that $\dstar$ and $\Delta$ are given at time $t$ simply
by scaling by $t$:
$\dstar(t) = t d^{\ast_h}$ and
$ \Delta(t) = t \Delta^h$.
Finally, we have that $\omega_T$ pulled back to $C_h$ is
\begin{equation*}
\omega_T = \Tr' (-1)^Q Q e^{-t \Delta^h} \frac{dt}{t},
\end{equation*}
which, modulo the regularizing factor $t^s$, is precisely the integrand in the zeta function defining the analytic torsion $T$.
Thus we have proven
\begin{lemma} \label{lemma: zeta as integral of omega}
	For $\operatorname{Re} s > \frac{n}{m}$,
	\begin{equation*}
	\zeta(s;  \Delta^h, (-1)^Q Q) =  \frac{1}{\Gamma(s)} \int_{C_h} t^s \omega_T,
	\end{equation*}
	where here $\omega_T$ denotes the one-form of \eqref{eq: def of omega T} pulled back to $C_h$.
\end{lemma}

Recalling the definition of analytic torsion $T(h)$ from \eqref{eq: def of torsion}
and using that $\frac{1}{\Gamma(s)} = s + O(s^2)$ as $s \to 0$ and Lemma \ref{lemma: zeta as integral of omega},
we have the heuristic
\begin{equation*}
2 \log T(h) =  \left. \left( \int_{C_h} t^s \omega_T \right) \right|^{\text{AC}}_{s=0},
\end{equation*}
where the superscript $\text{AC}$ indicates that we must analytically continue
the function in parentheses, which is only defined for $\Real s$ large.
(This heuristic is only literally true if the aforementioned analytic continuation is holomorphic at the origin.)
This is the sense in which we mean the following:
\begin{displayquote}
	The analytic torsion $\log T(h)$ may be interpreted as a regularized integral of $\frac{1}{2} \omega_T$ over the curve $C_h$ in the space of metrics $\M$.
\end{displayquote}

\begin{remark} \label{remark: mathai-wu}
	Evidently, this interpretation relies on the $\mathbb{Z}$-grading
	and does not make sense in the more general $\mathbb{Z}_2$-graded context
	of Mathai-Wu \cites{mathai-wu-twisted-2011, mathai-wu}.
	In fact, however, this interpretation does apply to $\mathbb{Z}_2$-graded
	elliptic complexes if one takes a more general approach.
	By considering the space of all inner products
	on the Hilbertable space $L^2(M,E)$,
	including \textit{non-local} inner products that do not arise from a metric on $E$,
	one can define more general versions of
	$b$ and $\omega_T$. 
	In that case, the regularized integral of $\omega_T$ along the curve corresponding to
	scaling the inner product on the subspace $\im d \subset L^2(M,E)$
	produces the Mathai-Wu torsion.
	We have chosen not to pursue this approach here.
\end{remark}

The preceding interpretation and the following theorem
are central to our viewpoint on the metric dependence of $T$.


\begin{theorem} \label{theorem: closed}
	The one-form $\omega_T$ is closed on $\M$, i.e., $\delta^\M \omega_T = 0$.
\end{theorem}

We will indicate the idea of the proof.
Fix an integer $N$ greater than $ \frac{n}{m}$
to  ensure that $R_z^N$ is trace-class,
where $R_z = (z-\Delta)^{-1}$ is the resolvent.
Define $\tilde{\omega}_T \in \Omega^1(\M, \mathbb{C})$
(whose dependence on $z$ we suppress) by
\begin{equation} \label{eq: def of omega T tilde}
\tilde{\omega}_T :=  \Tr (-1)^Q R_z^N b.
\end{equation}
The Cauchy integral formula
relates $\omega_T = \Tr' (-1)^Q e^{-\Delta} b$ and $\tilde{\omega}_T$,
and to prove that $\omega_T$ is closed, it is
sufficient to prove:

\begin{lemma}  \label{lemma: omega T tilde is closed}
	The one-form $\tilde{\omega}_T =  \Tr (-1)^Q R_z^N b $ is closed on $\M$.
\end{lemma}

To prove this lemma,
one may use
the familiar formula $\delta^\M R_z  =  R_z \,\left( \delta^{\M} \Delta \right) \, R_z$
for the derivative of the resolvent;
Lemmas \ref{lemma: derivative of d star and Delta}, \ref{lemma: b is symmetric}, and \ref{lemma: derivative of b}; and properties of the trace and adjoint
from Lemma \ref{lemma: adjoint}.
We omit the details here since Theorem \ref{theorem: multi closed} is a more general result (see Remark \ref{remark: m=1})
whose proof does not depend on this result.

\begin{remark} \label{remark: bismut-zhang}
	Our one-form $\omega_T$ is a generalization
	of the one-form  $\alpha$ introduced by Bismut-Zhang
	in the de Rham setting \cite{bismut-zhang}.
	To explain this briefly, let $g^{TM}$ be a Riemannian metric on $M$
	and let	$F$ be a flat bundle with metric $g^F$.
	Let $f: M \to \mathbb{R}$ be a Morse function.
	For $t \in (0, \infty)$ and $T \in \mathbb{R}$,
	consider the metric $h_{t,T}$ on $\Lambda^\bullet F$
	induced by the metrics $\frac{1}{t} g^{TM}$ (as in Remark \ref{remark: scale Riemannian metric}) and $e^{-2Tf} g^F$.
The pullback of $\omega_T$ to the surface in $\M$
parametrized by $(t,T) \mapsto h_{t,T}$ is
	\begin{equation*}
	 \Tr (-1)^Q Q e^{-t \Delta^T} \, \frac{dt}{t} ~ - ~ 2 \Tr (-1)^Q f e^{-t \Delta^T} \, dT,
	\end{equation*}
	which is Bismut-Zhang's $\alpha$.
	(Here, $\Delta^T$ denotes a conjugate of the 
	Witten Laplacian \cite{witten}.)
	Bismut-Zhang's proof of the generalized Cheeger-M\"{u}ller theorem
	uses that $\alpha $ is closed on  $\Sigma$, although
	they do not make our more general observation that $\alpha$
	is the pullback to $\Sigma$ of a closed form defined on all of $\M$.
\end{remark}

\subsection{Variation formula for torsion}
\label{subsection: proof of metric anomaly}

The following theorem
due to Schwarz \cite{schwarz}
is a generalization of Ray-Singer's fundamental result
for de Rham torsion \cite{ray-singer-71}.
This is a special case of Mathai-Wu's result for $\mathbb{Z}_2$-graded torsion,
Theorem 4.1 of \cite{mathai-wu}.
We will present a proof
using a Stokes' theorem argument
involving the closed one-form $\omega_T$.

\begin{theorem} \label{theorem: variation of T}
	Let $h(u)$, $u \in [0,1]$, be a smooth one-parameter family of metrics in $\M$.
	Let $T(h(u)) = T(M, E, h(u))$ be the associated analytic torsion.
	Then the derivative of $\log T(h(u))$
	is the following constant
	term in an asymptotic expansion:
	\begin{equation} \label{equation: derivative of torsion 2}
	\frac{d}{du} \log T(h(u)) = \left[ \Tr' (-1)^Q h(u)^{-1} \frac{d h}{d u} e^{-t \Delta^{h(u)}} \right]_{t^0} .
	\end{equation}
\end{theorem}

\begin{proof}
It suffices to compute the difference in analytic torsions associated to $h(0)$ and $h(1)$
	(thereby using the usual Stokes' theorem rather than an infinitesimal version).
	
	\label{proof of metric anomaly}
	For $A > \e > 0$, let $\Sigma = \Sigma_{A, \e}$ denote the surface  in $\M$ parametrized by $(u, t) \in [0,1] \times [\e, A] \mapsto h = h(u)_t := t^Q h(u)$.
	The boundary of $\Sigma$ consists of four curves,
	described respectively by $u = 0$, $u=1$, $t = \e$, and $t = A$.
	$\omega_T$ pulled back to $\partial \Sigma$, in the coordinates $(u, t)$, is
	\begin{align*}
	\omega_T =&  \Tr' (-1)^Q  Q e^{-t \Delta(u)} \frac{dt}{t} + \Tr' (-1)^Q e^{-t \Delta(u)} h^{-1}\frac{\partial h}{\partial u} du.
	\end{align*}
	Since $\omega_T$ is closed, $\delta^\M (t^{s} \omega_T)$ pulled back to $\Sigma$ is
	\begin{align*}
	\delta^\M (t^{s} \omega_T)
	&=  s t^{s - 1}   \Tr' (-1)^Q e^{- t \Delta(u)} h^{-1}\frac{\partial  h}{\partial u}  \, dt \, du .
	\end{align*}
	Let $\zeta_u(s) := \zeta(s; \Delta^{h(u)}, (-1)^Q Q)$.
	Let $\zeta_{u, A, \e}(s)$ be defined for $\Real s > \frac{n}{m}$ by
	\begin{equation*}
	\Gamma(s) \zeta_{u, A, \e}(s) :=  \int_{C_{h(u), A, \e}} t^s \omega_T .
	\end{equation*}
	Stokes' theorem gives, for $\Real s > \frac{n}{m}$,
	\begin{align*}
	\Gamma(s) \left[ \zeta_{0, A, \e}(s) - \zeta_{1, A, \e}(s) \right] =& \int_0^1 \int_\e^A  s t^{s - 1}   \Tr' (-1)^Q e^{- t \Delta(u)} h^{-1}\frac{\partial  h}{\partial u}  \, dt \, du \\
	&+ \int_0^1 t^s  \Tr' (-1)^Q e^{-\epsilon \Delta(u)}  h^{-1}\frac{\partial h}{\partial u} du \\
	&- \int_0^1 t^s  \Tr' (-1)^Q e^{-A \Delta(u)}  h^{-1}\frac{\partial h}{\partial u} du. 
	\end{align*}
	In the limits $A \to \infty$ and $\e \to 0$, the latter two terms tend to zero for $\Real s$ large
	by the admissibility conditions \ref{A1} and \ref{A2}.
	Thus we obtain, for $\Real s > \frac{n}{m}$,
	\begin{align*}
	\zeta_{0}(s) - \zeta_{1}(s) &= \frac{1}{\Gamma(s)}  \int_0^1 \int_0^\infty   s t^{s - 1}   \Tr' (-1)^Q e^{- t \Delta(u)} h^{-1}\frac{\partial  h}{\partial u}  \, dt \, du \\
	&=  \int_0^1  s  \zeta\left(s; \Delta(u),  (-1)^Q h^{-1}\frac{\partial  h}{\partial u} \right)  \, du.
	\end{align*}
	We have shown that this equality holds for $\Real s$ large, but both sides possess unique meromorphic continuations,
	so in fact the equality holds for all $s$.
	Differentiating and evaluating at $s=0$, we obtain
	\begin{align*}
	\log T(h(0)) - \log T(h(1)) &=  \int_0^1   \left. \frac{d}{ds} \right|_{s=0} \left( s  \zeta\left(s; \Delta(u),  (-1)^Q h^{-1}\frac{\partial  h}{\partial u} \right) \right)  \, du \\
&= \int_0^1 	 \zeta\left(0; \Delta(u),  (-1)^Q h^{-1}\frac{\partial  h}{\partial u} \right) \, du \\
&= \int_0^1 \left[\Tr' (-1)^Q e^{- t \Delta(u)} h^{-1}\frac{\partial  h}{\partial u} \right]_{t^0} \, du,
		\end{align*} 
		where we have used  Theorem \ref{theorem: L zeta properties}.
	This proves the theorem.
\end{proof}

From Theorem \ref{theorem: variation of T}
and Corollary \ref{coro: a'_0 vanishes}, one immediately obtains:
\begin{coro} \label{coro: n odd, acyclic}
	If the dimension $n$ is odd and $E$ is acyclic,
	then the analytic torsion $T(h)$ is independent of the metric $h \in \M$.
\end{coro}

\begin{remark}
	If $n$ is even and/or $E$ is not acyclic, then $T(h)$ may depend on the metric $h$, but
	Theorem \ref{theorem: variation of T}
	can be used to show the metric invariance
	of the Ray-Singer metric on the determinant line of the cohomology $H^\bullet(M, E)$
	(in the odd-dimensional case)
	or of a relative analytic torsion
	(in the even-dimensional case).
	See, e.g., \cite{mathai-wu} for a review of the details.
\end{remark}


\section{Definitions of multi-zeta functions and multi-torsion}
\label{section: multi-torsion}

\subsection{Multi-zeta functions}
\label{subsection: Multi-zeta functions}
Defining multi-torsion will require
developing a theory of multi-zeta functions.
As motivation,
let us consider a product manifold.
For $j=1,2$, let $M_j$ be a  compact manifold of dimension $n_j$;
denote by $\pi_j$ the projection $M_1 \times M_2 \to M_j$.
For $j=1,2$, let $E_j \to M_j$ be a vector bundle
and let $L_j$
be a strictly positive symmetric elliptic operator on sections of $E_j$.
Consider the product vector bundle over the product manifold:
$E := \pi_1^\ast E_1 \otimes \pi_2^\ast E_2 \to M_1 \times M_2$.
$L_1$ and $L_2$ induce commuting sub-elliptic operators,
which we will also denote by $L_1$ and $L_2$, on sections of $E$.
Define $L(t_1, t_2):= t_1 L_1  +  t_2 L_2$, which is an elliptic operator depending on two positive parameters $t_1$ and $t_2$.
For $j=1,2$,
let $h_j(t_j) := \Tr e^{-t_j L_j}$, and let $\zeta_j(s_j)$
be the associated zeta function.
Let $h(t_1, t_2):=   \Tr e^{-L(t_1, t_2)}$, which decomposes as
a product:
$h(t_1, t_2) = h_1(t_1) h_2(t_2)$.
Define the ``multi-zeta function'' $\zeta(s_1, s_2)$
by $\zeta(s_1, s_2) := \zeta_1(s_1) \zeta_2(s_2)$, which
by Fubini's theorem
may be written as follows,
for $\Real s_1, \Real s_2$ large:
\begin{equation} \label{eq: multi-zeta in product case}
\zeta(s_1, s_2) = \frac{1}{\Gamma(s_1) \Gamma(s_2)}  \int_0^\infty \int_0^\infty t_1^{s_1 - 1} t_2^{s_2 - 1} \,  h(t_1, t_2) \, dt_1 dt_2.
\end{equation}
Note also that we have the following relationship
between derivatives at the origin:
\begin{align} \label{eq: derivative of product}
\left. \frac{\partial^2 \zeta}{\partial s_1 \partial s_2} \right|_{(s_1, s_2)=(0,0)}
&= \left. \frac{d \zeta_1}{ds_1} \right|_{s_1=0} \left. \frac{d \zeta_2}{ds_2} \right|_{s_2=0} .
\end{align}

We may also consider more general functions $h(t_1, t_2)$
and define an associated multi-zeta function
as in \eqref{eq: multi-zeta in product case}.
In the product case, ``multiplication" of the respective properties of the admissible functions
$h_1(t_1)$ and $h_2(t_2)$ (recall Definition \ref{defn: admissible})
implies that their product
has the following four ``multi-asymptotic" properties:
\begin{enumerate}[label={(MA\arabic*)}]
	
	\item {\em (Both $t_1$ and $t_2$ small.)} \label{first}
	For every real number $K$,
	we may write
	\begin{align*}
	h(t_1, t_2) 
	= & \sum_{j_1 = 0}^{J^1_K} \sum_{j_2 = 0}^{J^2_K} \, (t_1)^{p^1_{j_1}} \, (t_2)^{p^2_{j_2}}  \, a_{p^1_{j_1}, p^2_{j_2} }  \\
	+ &\sum_{i_1 = 0}^{I^1_K} \, (t_1)^{p^1_{1_1}} \, b^1_{p^1_{i_1}}(t_2)  + \sum_{i_2 = 0}^{I^2_K} \, (t_2)^{p^2_{i_2}} \, b^2_{p^2_{i_2}}(t_1)  
	+ r_K(t_1, t_2),
	\end{align*}
	where for $0 < t_1, t_2 < 1$, $|r_K(t_1, t_2)| \leq C_K (t_1)^K (t_2)^K$, $| b^1_{p^1_{i_1}}(t_2) | \leq C_{K,i_1} (t_2)^K$,\
	and $| b^2_{p^2_{i_2}}(t_1) | \leq C_{K,i_2} (t_1)^K$.

	\item {\em ($t_1$ small, $t_2$ large.)} \label{second}
	For every real number $K$ we may write  
	\begin{equation*} 
	h(t_1, t_2)
	=  \sum_{i_1 = 0}^{I^1_K}  \, (t_1)^{ p^1_{i_1}} \, c^1_{p^1_{i_1}}(t_2) +  s^1_K(t_1, t_2) 
	\end{equation*}
	where for $0< t_1 < 1$ and $t_2 \geq 1$,
	$| c^1_{p^1_{i_1}}(t_2) | \leq  C_{K, i_1}e^{-\epsilon_{K, i_1} t_2 } $ and $|s^1_K(t_1, t_2)| \leq C_K (t_1)^K e^{-\epsilon_K t_2} $.

	\item {\em ($t_1$ small, $t_2$ large.)} \label{third}
	For every real number $K$ we may write  
	\begin{equation*} 
	h(t_1, t_2)
	=  \sum_{i_2 = 0}^{I^2_K}  \, (t_2)^{ p^2_{i_2}} \, c^2_{p^2_{i_2}}(t_1) +  s^2_K(t_1, t_2) 
	\end{equation*}
	where for $0< t_2 < 1$ and $t_1 \geq 1$,
	$| c^2_{p^2_{i_2}}(t_1) | \leq C_{K, i_2}e^{-\epsilon_{K, i_2} t_1 }$ and $|s^2_K(t_1, t_2)| \leq C_K (t_2)^K e^{-\epsilon_K t_1} $.

	\item {\em (Both $t_1$ and $t_2$ large.)} If $t_1, t_2 \geq 1$, \label{last}
	\begin{equation*}
	|h(t_1, t_2)| \leq C e^{-\epsilon_1 t_1} e^{-\epsilon_2 t_2}.
	\end{equation*}

\end{enumerate}

\begin{defn} \label{defn: multi-admissible}
	We will say that a smooth function $(0, \infty) \times (0, \infty) \to \mathbb{C}$
	is multi-admissible when it satisfies conditions \ref{first}-\ref{last} above.
	We then define the multi-zeta function associated to 
	a multi-admissible function
	$h(t_1, t_2)$ by the following, for 
	$\Real s_1$ and $\Real s_2$ large:
	\begin{equation} \label{def of multi-zeta}
	\zeta(s_1, s_2) :=  \frac{1}{\Gamma(s_1) \Gamma(s_2)} \int_0^\infty \int_0^\infty 
	t_1^{s_1} t_2^{s_2} \, h(t_1, t_2) \, \frac{dt_1}{t_1} \frac{dt_2}{t_2}.
	\end{equation}
\end{defn}

\begin{remark}
	It is clear that the product of an admissible
	function of $t_1$ and an admissible function of $t_2$,
	or more generally a finite sum of such products,
	is multi-admissible.
\end{remark}

Lemma \ref{lemma: zeta properties} generalizes to:

\begin{lemma} \label{lemma: multi-zeta properties}
	Suppose $h(t_1, t_2)$ is multi-admissible. Then the 
	integral defining $\zeta(s_1, s_2)$ in \eqref{def of multi-zeta} converges for $\Real s_1 > - p^1_0$ and $\Real s_2 > - p^2_0$
	and defines a holomorphic function there.
	This function admits a unique meromorphic extension to all of $\mathbb{C}^2$;
	the extension, which we also denote by $\zeta(s_1, s_2)$,
	is holomorphic at the origin, where its value is
	$  \zeta(0,0)  = a_{0,0} $.
	Furthermore, we have the vanishing results:
	\begin{enumerate}
		\item If no $(t_1)^0$ terms appear in \ref{first} or in \ref{second},
		then $ \frac{\partial \zeta}{\partial s_2}(0,0) = 0$.
		\item If no $(t_2)^0$ terms appear in \ref{first} or in \ref{third},
		then $ \frac{\partial \zeta}{\partial s_1}(0,0) = 0$.
	\end{enumerate}
\end{lemma}

\begin{proof}
	Let $f(s_1, s_2)$ be the integral in \eqref{def of multi-zeta}:
	\begin{equation*}
	f(s_1, s_2) := \int_0^\infty \int_0^\infty 
	t_1^{s_1} t_2^{s_2} \, h(t_1, t_2) \, \frac{dt_1}{t_1} \frac{dt_2}{t_2},
	\end{equation*}
	so that
	\begin{equation} \label{zeta-gamma-f}
	\zeta(s_1, s_2) = \frac{1}{\Gamma(s_1) \Gamma(s_2)} f(s_1, s_2).
	\end{equation}
	The gamma function $\Gamma$ is meromorphic and nonvanishing, so
	$\frac{1}{\Gamma(s_1) \Gamma(s_2)}$ is holomorphic everywhere.
	Since $\Gamma(z) = \frac{1}{z} + O(1)$ as $z \to 0$,
	\begin{equation} \label{gamma-at-origin}
	\frac{1}{\Gamma(s_1) \Gamma(s_2)} = s_1 s_2 + O(|(s_1, s_2)|^3) \text{ as } (s_1, s_2) \to (0,0).
	\end{equation}
	Thus it suffices to show that $f(s_1, s_2)$ admits a meromorphic extension
	and to study its behavior near $(0,0)$.
	
	We may write the domain $	(0,\infty) \times (0, \infty) $
	of the integral in \eqref{def of multi-zeta}
	as the union of the four sets
	$(0,1] \times (0,1]$, $(0,1] \times [1, \infty)$, 
	$ [1, \infty) \times (0,1]$, and $ [1, \infty) \times  [1, \infty)$.
	The respective conditions \ref{first}-\ref{last} ensure that the integral
	over each of these four sets converges absolutely if $\Real s_j > \frac{n_j}{2}$ for  $j=1,2$.
	Computing the integrals give that for any real number $K$ we have
	\begin{align*}
	f(s_1, s_2)
	=& \sum_{j_1 = 0}^{J^1_K} \sum_{j_2 = 0}^{J^2_K} \, \frac{1}{(s_1 - p^1_{j_1})(s_2 - p^2_{j_2})}  \, a_{p^1_{j_1}, p^2_{j_2} }  \\
	&+ \sum_{i_1 = 0}^{I^1_K} \, \frac{1}{s_1 - p^1_{i_1}} \, \tilde{b}^1_{p^1_{i_1}}(s_2) ~  + ~ \sum_{i_2 = 0}^{I^2_K} \, \frac{1}{s_2 - p^2_{i_2}} \, \tilde{b}^2_{p^2_{i_2}}(s_1)   \\
	&+ \tilde{r}_K(s_1, s_2) \\
	&+  \sum_{i_1 = 0}^{I^1_K}  \, \frac{1}{s_1 - p^1_{i_1}} \, \tilde{c}^1_{p^1_{i_1}}(s_2) ~ +~   \tilde{s}^1_K(s_1, s_2)  \\
	&+  \sum_{i_2 = 0}^{I^2_K}  \, \frac{1}{s_2 -  p^2_{i_2}} \, \tilde{c}^2_{p^2_{i_2}}(s_2)~ +~  \tilde{s}^2_K(s_1, s_2)  \\
	&+ \tilde{h}(s_1, s_2),
	\end{align*}
	where the new functions of $s_1, s_2$ (marked by tildes), each of which is holomorphic for $\Real s_1, \Real s_2 > -K$,
	are defined in terms of the functions in \ref{first}-\ref{last} by
	\begin{align*}
	\tilde{b}^1_{p^1_{i_1}}(s_2) : =&  \int_0^1 (t_2)^{s_2} \, b_{p^1_{i_1}}(t_2) \, \frac{dt_2}{t_2}; \\
	\tilde{b}^2_{p^2_{i_2}}(s_1) : =&  \int_0^1 (t_1)^{s_1} \, b_{p^2_{i_2}}(t_1) \, \frac{dt_1}{t_1}; \\
	\tilde{r}_K(s_1, s_2) :=& \int_0^1 \int_0^1 (t_1)^{s_1} (t_2)^{s_2} \, r_K(t_1, t_2) \, \frac{dt_1}{t_1} \frac{dt_2}{t_2};  \\
	\tilde{c}^1_{p^1_{i_1}}(s_2) :=& \int_1^\infty t_2^{s_2} \,c_{p^1_{i_1}}(t_2) \, \frac{dt_2}{t_2}; \\
	\tilde{s}^1_K(s_1, s_2) :=& \int_1^\infty \int_0^1  (t_1)^{s_1} (t_2)^{s_2} \, s^1_K(t_1, t_2) \, \frac{dt_1}{t_1} \frac{dt_2}{t_2}; \\
	\tilde{c}^2_{p^2_{i_2}}(s_1) :=& \int_1^\infty t_1^{s_1} \,c_{p^2_{i_2}}(t_1) \, \frac{dt_1}{t_1}; \\
	\tilde{s}^2_K(s_1, s_2) :=&  \int_0^1 \int_1^\infty  (t_1)^{s_1} (t_2)^{s_2} \, s^2_K(t_1, t_2) \, \frac{dt_1}{t_1} \frac{dt_2}{t_2}; \\
	\tilde{h}(s_1, s_2) :=& \int_0^1 \int_1^\infty  (t_1)^{s_1} (t_2)^{s_2} \, h(t_1, t_2) \, \frac{dt_1}{t_1} \frac{dt_2}{t_2}.
	\end{align*}
	This uniquely defines the meromorphic extension of $f$ to the set where $\Real s_1, \Real s_2 > -K$.
	Since $K$ is arbitrary, this proves that $f$ has a unique meromorphic extension to all of $\mathbb{C}^2$.
	We see that as $(s_1, s_2) \to (0,0)$,
	\begin{equation} \label{f(0,0)}
	f(s_1, s_2) = a_{0,0} \frac{1}{s_1 s_2} + \left( \tilde{b}^1_{0}(0) + \tilde{c}^1_{0}(0) \right) \frac{1}{s_1} + 
	\left( \tilde{b}^2_{0}(0) + \tilde{c}^2_{0}(0) \right)\frac{1}{s_2} + O\left( 1 \right).
	\end{equation} 
	From \eqref{zeta-gamma-f}, \eqref{gamma-at-origin}, 
	and \eqref{f(0,0)}, we see that
	\begin{align*}
	\zeta(0,0) &= a_{0,0};
	\frac{\partial \zeta}{\partial s_2}(0,0) = \tilde{b}^1_{0}(0) + \tilde{c}^1_{0}(0); 
	\frac{\partial \zeta}{\partial s_1}(0,0) = \tilde{b}^2_{0}(0) + \tilde{c}^2_{0}(0).
	\end{align*}
	The final vanishing assertions of the lemma follow.
\end{proof}

\begin{remark}
	There is a natural generalization to functions $h(t_1, \dots, t_m)$ of $m$ positive real variables,
	for which one may define a multi-zeta function $\zeta(s_1, \dots, s_m)$
	of $m$ complex variables.
	We will not pursue this further here.
\end{remark}

\begin{remark}
	The multi-zeta functions we consider here bear a formal resemblance to the zeta functions of Shintani \cite{shintani},
	but we do not know if there is any deeper connection.
\end{remark}

\subsection{Quotients of product manifolds}
\label{subsection: The geometry of finite quotients of product manifolds}

To establish some notational conventions,
let $\tilde{V} \to \tilde{N}$ and $V \to N$ be vector bundles
over manifolds $\tilde{N}$ and $N$.
By a vector bundle isomorphism,
we will mean a diffeomorphism $\phi: \tilde{N} \to N$ 
that lifts to a diffeomorphism (also called $\phi$) $\phi: \tilde{V} \to V$
between the total spaces
that acts as a linear isomorphism on fibers, i.e.,
for every $\tilde{x} \in \tilde{N}$, $\phi_{\tilde{x}}: \tilde{V}_{\tilde{x}} \to V_{x}$ is a linear isomorphism,
where $x = \phi(\tilde{x}) \in N$.
Such a $\phi$ induces a pullback operator $\phi^\ast: C^\infty(N, V) \to C^\infty(\tilde{N}, \tilde{V})$
defined by $\phi^\ast(\sigma)(\tilde{x}) = (\phi_x)^{-1} \sigma(x)$.
This extends also to a pullback operator $\phi^\ast:  C^\infty(N, \Lambda^\bullet V) \to C^\infty(\tilde{N}, \Lambda^\bullet \tilde{V})$.
We will write $\phi_{\tilde{x}}^\ast$
for the isomorphism of fibers $\phi_{\tilde{x}}^\ast: (\Lambda^\bullet V)_x \to (\Lambda^\bullet \tilde{V})_{\tilde{x}}$.

In the case when $\tilde{V}$ and $V$ are both flat, we
will say that a vector bundle isomorphism $\phi$ is flat if 
$d^{\Lambda^\bullet \tilde{V} } \circ \phi^\ast = \phi^\ast \circ d^{\Lambda^\bullet V} $,
where $d^{\Lambda^\bullet \tilde{V}}$ and  $d^{\Lambda^\bullet V}$ 
denote the respective exterior derivatives induced by the flat connections.
If in addition $\tilde{V}$ and $V$ are orthogonal or unitary flat bundles,
then we will say $\phi$ is a flat isometry if for every $\tilde{x} \in \tilde{N}$,
$\phi_x: \tilde{V}_{\tilde{x}} \to V_{x}$ is an orthogonal or unitary map.

Let $\operatorname{Diff}(\tilde{N}, N)$ denote the set of diffeomorphisms
$\tilde{N} \to N$; let $\operatorname{Diff}(N) := \operatorname{Diff}(N,N)$,
which forms a group.
Let $\operatorname{Isom}(\tilde{V} \to \tilde{N}, V \to N)$
denote the set of flat isometries; let $\operatorname{Isom}(V \to N):= \operatorname{Isom}(V \to N, V \to N)$, which forms a group.

Now we may describe the geometric setting in which we will define multi-torsion.
Suppose that $M_1$ and $M_2$
are compact oriented manifolds and  that $F_1 \to M_1$ and $F_2 \to M_2$ are orthogonal
or unitary flat vector bundles.
Let $F = \pi_1^\ast F_1 \otimes \pi_2^\ast F_2 \to M_1 \times M_2$
be the flat product vector bundle,
where for $j=1,2$, $\pi_j: M_1 \times M_2 \to M_j$ denotes the projection.
Let $\Gamma$ be a finite subgroup of $\operatorname{Isom}(F_1 \to M_1) \times \operatorname{Isom}(F_2 \to M_2)$,
which we view as a subgroup of $\operatorname{Isom}(F \to M_1 \times M_2)$.
For $j=1,2$, let $\Gamma_j$ denote the projection of $\Gamma$ onto $\operatorname{Isom}(F_j \to M_j)$.
We remark that each $\Gamma_j$ is a subgroup of $\operatorname{Isom}(F_j \to M_j)$,
but $\Gamma$ need not be the product group $\Gamma_1 \times \Gamma_2$.

We will assume $\Gamma$, viewed as a group of diffeomorphisms of $M_1 \times M_2$,
acts freely.
This means for every $\gamma = (\gamma_1, \gamma_2) \in  \Gamma \setminus \{ I \}$,
$\gamma$ has no fixed points as a diffeomorphism of $M_1 \times M_2$,
i.e., $\gamma_1$ and $\gamma_2$ cannot both have a fixed point.
But we will allow exactly one (or zero) of $\gamma_1$ and $\gamma_2$ to have fixed points.
Recall that a fixed point $x \in N$ of $\phi: N\to N$ is said to be nondegenerate when
$I - d\phi_x: T_x N \to T_x N$ is an isomorphism.

Let $M_\Gamma$ be the quotient of $M_1 \times M_2$
by $\Gamma$,
which is a smooth manifold
of dimension $n := n_1 + n_2$.
The action of $\Gamma$ on $F$ induces a quotient
orthogonal or unitary flat vector bundle $F_\Gamma \to M_\Gamma$.
Furthermore, for every $x \in M_\Gamma$, there is an open neighborhood $U \subset M_\Gamma$ containing $x$
such that $U$ is diffeomorphic to a product $\tilde{U}_1 \times \tilde{U}_2$
for some $\tilde{U}_1$ and $\tilde{U}_2$, open sets in $M_1$ and $M_2$, respectively,
and the restriction of $F_\Gamma$ to $U$ is flat-isometric with
the product of the restrictions $\pi_{\tilde{U}_1}^\ast F_1|_{\tilde{U}_1} \otimes \pi_{\tilde{U}_2}^\ast F_2|_{\tilde{U}_2}$.

This local product structure gives a bigrading of $F_\Gamma$-valued forms.
Let $E_\Gamma := \Lambda^\bullet F_\Gamma \to M_\Gamma$.
Then $E_\Gamma = \bigoplus_{q_1,q_2} E_\Gamma^{q_1,q_2}$,
where we may locally identify $E_\Gamma^{q_1,q_2}$ with $(\pi_1^\ast F_1 \otimes \Lambda^{q_1} T^\ast M_1) \otimes (\pi_2^\ast F_2 \otimes \Lambda^{q_2} T^\ast M_2)$.
We will refer to sections of $E_\Gamma^{q_1,q_2}$ as $(q_1, q_2)$-forms.
For $j=1,2$, we will use $Q_j$ 
to denote the operator that acts on $(q_1,q_2)$-forms
by multiplication by $q_j$.
Let $Q := Q_1 + Q_2$; note that $(-1)^Q = (-1)^{Q_1} (-1)^{Q_2}$.

Let $d$ be the exterior derivative on sections of $E_\Gamma$.
The local product structure ensures that there is a global decomposition of $d$:
$d = d_1 + d_2$,
where $d_1$ maps $(q_1, q_2)$-forms to $(q_1 + 1, q_2)$-forms,
and $d_2$ maps $(q_1, q_2)$-forms to $(q_1, q_2 + 1)$-forms.

To obtain furthermore a global decomposition of $d^\ast$ and $\Delta$,
we must consider metrics that induce a local {\em geometric} product structure on $M_\Gamma$.
Let $\M_{\Gamma_j}$ denote the set of $\Gamma_j$-invariant Riemannian metrics on $M_j$.
A choice of metrics $h_1 \in \M_{\Gamma_1}$ and $h_2 \in \M_{\Gamma_2}$
induces a product Riemannian metric $h_1 \times h_2$ on $M_1 \times M_2$.
Let $\Mprod$ denote the set of metrics arising in this way;
$\Mprod$ is identifiable with $\M_{\Gamma_1} \times \M_{\Gamma_2}$.

Consider a metric $h = h_1 \times h_2 \in \Mprod$;
$h$ is $\Gamma$-invariant,
and therefore descends to a well-defined metric, which we will also call $h$, on the quotient $M_\Gamma$.
Furthermore, $h$ induces a local geometric product structure in the sense that
the identification $U \cong \tilde{U}_1 \times \tilde{U}_2$ from above is a Riemannian isometry,
where $\tilde{U}_1 \times \tilde{U}_2 \subset M_1 \times M_2$ is given the product metric induced by $h_1$ and $h_2$.
Under these hypotheses we have that $d^\ast = d^\ast_1 + d^\ast_2$ and
\begin{align}
\Delta &= \Delta_1 + \Delta_2, \text{ where } \label{equation: decomposition of Laplacian} 
\Delta_j = d_j d^\ast_j + d^\ast_j d_j \text{ for } j=1,2.
\end{align}
Our assumptions imply the vanishing of the following (anti)commutators:
\begin{gather*}
\{d_1, d_2\} = 0,
\{d_1^\ast, d_2^\ast\} = 0,
\{d_1, d^\ast_2\} = 0,
\{d_2, d^\ast_1\} = 0, \\
[d_i, \Delta_j] = 0 \text{ (for $i,j =1, 2$)},
[d^\ast_i, \Delta_j] = 0 \text{ (for $i,j =1, 2$)}, \text{ and }
[\Delta_1, \Delta_2] = 0 .
\end{gather*}

\begin{remark}
	We will only consider metrics on $E_\Gamma$
	that are induced by Riemannian metrics on the product $M_1 \times M_2$
	and by the canonical metric on $F_\Gamma$.
	Also, unlike in previous sections,
	we will \textit{not} fix a density on $M_\Gamma$;
	rather, we will always take the Riemannian density induced by the
	product Riemannian metric.
\end{remark}


\subsection{Multi-torsion}
\label{subsection: Definition of multi-torsion}

Consider the generalized Laplacian 
on $F_\Gamma$-valued forms defined by, for $t_1, t_2 > 0$,
\begin{equation*}
\Delta(t_1, t_2) := t_1 \Delta_1 + t_2 \Delta_2.
\end{equation*}
$\Delta(t_1, t_2)$ is elliptic and nonnegative,
so its associated heat operator $e^{-\Delta(t_1, t_2)}$ is defined.
We have the following fundamental result,
which we will prove in \S \ref{subsection: proof of heat kernel multi-asymptotics}:
\begin{theorem} \label{theorem: heat kernel is multi-admissible}
	Suppose both $F_1 \to M_1$ and $F_2 \to M_2$ are acyclic.
	For $j=1,2$, let $\alpha_j \in C^\infty(M_j,  \End \Lambda^\bullet F_j )$, and suppose that $\alpha_j$ commutes with the action of  $\Gamma_j$,
	inducing an operator $\alpha = \alpha_1 \alpha_2 := \alpha_1 \otimes \alpha_2 $ that commutes with the action of $\Gamma$ and descends to an operator $\alpha_\Gamma \in C^\infty(M_\Gamma, \End  E_\Gamma)$.
	Then $\Tr \alpha_\Gamma e^{-\Delta(t_1, t_2)} $ is multi-admissible.
\end{theorem}

This allows us to define the associated multi-zeta function as in \eqref{def of multi-zeta}, i.e.:
\begin{defn}
	Suppose both $F_1 \to M_1$ and $F_2 \to M_2$ are acyclic.
	Let $\alpha_\Gamma$ be as above.
	For $\Real s_1 > \frac{n_1}{2}$ and $\Real s_2 > \frac{n_2}{2}$, define $\zeta(s_1, s_2; \Delta(t_1, t_2), \alpha_\Gamma)$ by the following:
	\begin{align} \label{def of heat multi-zeta}
	\zeta(s_1, s_2; \Delta(t_1, t_2), \alpha_\Gamma) 
	:=  \frac{1}{\Gamma(s_1) \Gamma(s_2)} \int_0^\infty \int_0^\infty 
	t_1^{s_1} t_2^{s_2} \Tr \alpha_\Gamma e^{-\Delta(t_1, t_2)}  \, \frac{dt_1}{t_1} \frac{dt_2}{t_2}.
	\end{align}
\end{defn}
\begin{remark}
	The assumption of acyclicity in both factors
	is essential to the decay of $\Tr \alpha_\Gamma e^{-\Delta(t_1, t_2)}$ in the limits $t_1 \to \infty$ (for fixed $t_2$)
	and $t_2 \to \infty$ (for fixed $t_1$),
	which ensure the integrability of $\Tr \alpha_\Gamma e^{-\Delta(t_1, t_2)}$.
	We have also imposed the strong assumption
	that the group $\Gamma$ is finite,
	but we expect that this assumption is less essential
	and hope that it could be weakened
	with more delicate analysis than we have pursued here.
\end{remark}

From Theorem \ref{theorem: heat kernel is multi-admissible} and Lemma \ref{lemma: multi-zeta properties}, we immediately obtain:
\begin{lemma} Under the assumptions above, the 
	integral defining $\zeta(s_1, s_2; \Delta(t_1, t_2), \alpha_\Gamma)$ in \eqref{def of heat multi-zeta}
	converges on the set of points $(s_1, s_2) \in \mathbb{C}^2$ such that
	$\Real s_1 > \frac{n_1}{2}$ and $\Real s_2 > \frac{n_2}{2}$
	and defines a holomorphic function there. Furthermore,
	$\zeta(s_1, s_2; \Delta(t_1, t_2), \alpha)$ admits a unique meromorphic extension to all of $\mathbb{C}^2$;
	the extension is holomorphic at the origin.
\end{lemma}

This allows us to make the:

\begin{defn} \label{definition: multi-torsion}
	For $(h_1, h_2) \in \Mprod$, we define multi-torsion $MT = MT(M_\Gamma, F_\Gamma, h_1, h_2)$ by
	\begin{align}
	MT(M_\Gamma, F_\Gamma, h_1, h_2)  := \frac{1}{4}  \left. \frac{\partial^2}{\partial s_1 \partial s_2} \right|_{(s_1, s_2)=(0,0)} \zeta(s_1, s_2; \Delta(t_1, t_2), (-1)^Q Q_1 Q_2). \label{eq: def of MT}
	\end{align}
	If we wish to emphasize only the dependence of the multi-torsion on the metrics $h_1$ and $h_2$,
	we will denote it by $MT(h_1, h_2)$.
\end{defn}

\begin{remark}
	The multi-zeta function
	in  \eqref{eq: def of MT} may be written as (for $\Real s_1 > \frac{n_1}{2}$ and $\Real s_2 > \frac{n_2}{2}$)
	\begin{equation}
	\begin{split} \label{equation: multi-torsion zeta function}
	&\zeta(s_1, s_2; \Delta(t_1, t_2), (-1)^Q Q_1 Q_2)  \\
	&= \frac{1}{\Gamma(s_1) \Gamma(s_2)} \int_0^\infty \int_0^\infty 
	t_1^{s_1} t_2^{s_2} \Tr (-1)^Q Q_1 Q_2 e^{-\Delta(t_1, t_2)} \, \frac{dt_1}{t_1} \frac{dt_2}{t_2} \\
	&= \frac{1}{\Gamma(s_1) \Gamma(s_2)} \int_0^\infty \int_0^\infty 
	t_1^{s_1} t_2^{s_2} \sum_{q_1, q_2} (-1)^{q_1 + q_2} q_1 q_2 \Tr e^{-(t_1 \Delta_1^{q_1} + t_2 \Delta_2^{q_2})} \, \frac{dt_1}{t_1} \frac{dt_2}{t_2},
	\end{split}
	\end{equation}
	where $t_1 \Delta_1^{q_1} + t_2 \Delta_2^{q_2} $ denotes the restriction of $t_1 \Delta_1 + t_2 \Delta_2$ to $(q_1, q_2)$-forms.
\end{remark}

\begin{remark}
	To motivate Definition \ref{definition: multi-torsion},
	let us consider the special case in which $\Gamma$ is the trivial group,
	i.e., $M_{\Gamma}$ is simply the product $M_1 \times M_2$
	and $F_\Gamma = \pi_1^\ast F_1 \otimes \pi_2^\ast F_2 \to M_1 \times M_2$.
	Then
	\begin{align*}
	\Tr (-1)^Q Q_1 Q_2 e^{-\Delta(t_1, t_2)}= \left( \Tr (-1)^{Q_1} Q_1 e^{-t_1 \Delta_1} \right) \left( \Tr (-1)^{Q_2} Q_2 e^{-t_2 \Delta_2} \right),
	\end{align*}
	and therefore  the multi-zeta function in \eqref{equation: multi-torsion zeta function} decomposes as a product:
	\begin{align*}
	\zeta(s_1, s_2; \Delta(t_1, t_2), (-1)^Q Q_1 Q_2)  
	&= \zeta(s_1; \Delta_1, (-1)^{Q_1}Q_1) \, \zeta(s_2; \Delta_2, (-1)^{Q_2}Q_2).
	\end{align*}
	Recalling Definition \ref{definition: torsion} of analytic torsion
	and the observation \eqref{eq: derivative of product}, we have that in this case
	the multi-torsion is the product of the logarithms of the analytic torsions of the factors:
	\begin{align*}
	MT(M_1 \times M_2, \pi_1^\ast F_1 \otimes \pi_2^\ast F_2, h_1, h_2  ) 
	&= \log T(M_1, \Lambda^\bullet F_1, h_1) \log T(M_2,  \Lambda^\bullet F_2, h_2).
	\end{align*}
	Thus in this case, if the dimensions $n_1$ and $n_2$ are both odd, then
	$MT(h_1, h_2)$ is trivially independent of the metrics $h_1$ and $h_2$ by Corollary \ref{coro: n odd, acyclic}
	since we have assumed $F_1$ and $F_2$ are both acyclic.
\end{remark}

Our goal is to study the dependence of $MT$ on the metric in the general case.
First, in even dimensions, we have the following vanishing result
generalizing Theorem 2.3 of \cite{ray-singer-71}:
\begin{theorem} \label{theorem: even vanishing of MT}
	Suppose that for at least one of either   $j=1$ or $j=2$,
	$n_j$ is even and $\Gamma_j$ acts on $M_j$ by orientation-preserving diffeomorphisms.
	Then $MT(h_1, h_2) = 0$ for any $(h_1, h_2) \in \Mprod$.
\end{theorem}
\begin{proof}
	Without loss of generality, we may assume $j=1$.
	The Hodge star on $M_1$ induces an operator $\ast_1$ on sections of $F \otimes \Lambda^\bullet T^\ast (M_1 \times M_2)$.
	Since $\Gamma_1$ acts by orientation-preserving isometries, $\ast_1$ descends to a well-defined invertible
	operator on the quotient, which we also denote by $\ast_1$, that maps $(q_1, q_2)$-forms to $(n_1 - q_1, q_2)$-forms.
	Thus we have the commutation relations
	\begin{align*}
	\ast_1 Q_1 &=  (n_1 I - Q_1)\ast_1, \\
	\ast_1 Q_2 &=  Q_2 \ast_1, \\
	\ast_1 (-1)^Q &= (-1)^{n_1} (-1)^Q \ast_1, \\
	\ast_1 \Delta(t_1, t_2) &= \Delta(t_1, t_2) \ast_1.
	\end{align*}
	The last implies that $\ast_1$ commutes also with $e^{-\Delta(t_1, t_2)}$.
	We have the following, using the cyclicity of the trace:
	\begin{align*}
	\Tr (-1)^Q Q_1 Q_2 e^{-\Delta(t_1, t_2)}  &=  \Tr \ast_1^{-1}(-1)^Q Q_1 Q_2 e^{-\Delta(t_1, t_2)} \ast_1 \\
	&= (-1)^{n_1} \Tr  (-1)^Q (n_1 I - Q_1) Q_2 e^{-\Delta(t_1, t_2)}.
	\end{align*}
	In the case when $n_1$ is even, this implies
	\begin{equation*} 
	\Tr (-1)^Q Q_1 Q_2 e^{-\Delta(t_1, t_2)}  = \frac{n_1}{2} \Tr  (-1)^Q Q_2 e^{-\Delta(t_1, t_2)},
	\end{equation*}
	which we claim vanishes. This vanishing would show that the zeta function of \eqref{equation: multi-torsion zeta function}
	vanishes identically, which would prove the proposition.
	
	To prove the claim that $\Tr  (-1)^Q Q_2 e^{-\Delta(t_1, t_2)} = 0$,
	we will use a version of McKean-Singer's argument
	relating the index to the heat equation \cite{mckean-singer}.
	We start by differentiating with respect to $t_1$.
	Since $ \frac{\partial}{\partial t_1} \Delta(t_1, t_2) = \Delta_1$,
	which commutes with $\Delta$ and therefore with the resolvent $(\Delta - z)^{-1}$,
	we obtain
	\begin{align*}
	&\frac{\partial}{\partial t_1} \Tr  (-1)^Q Q_2 e^{-\Delta(t_1, t_2)} \\
	&= - \Tr  (-1)^Q Q_2 e^{-\Delta(t_1, t_2)} \Delta_1 \\
	&= - \Tr  (-1)^Q Q_2 e^{-\Delta(t_1, t_2)} d_1 d^\ast_1 - \Tr  (-1)^Q Q_2 e^{-\Delta(t_1, t_2)}  d^\ast_1 d_1.
	\end{align*}
	But this is zero for the following reason: the cyclicity of the trace
	and the fact that $d_1$ commutes with $Q_2$ and $e^{-\Delta(t_1, t_2)}$ and anticommutes with $(-1)^Q$
	imply that
	\begin{equation*}
	\Tr  (-1)^Q Q_2 e^{-\Delta(t_1, t_2)} d_1 d^\ast_1 = - \Tr (-1)^Q Q_2 e^{-\Delta(t_1, t_2)}  d^\ast_1 d_1.
	\end{equation*}
	We have shown that for each $t_2 >0$, $\Tr  (-1)^Q Q_2 e^{-\Delta(t_1, t_2)}$ is constant in $t_1$.
	But Theorem \ref{theorem: heat kernel is multi-admissible} gives that $\Tr  (-1)^Q Q_2 e^{-\Delta(t_1, t_2)} \to 0$ as $t_1 \to \infty$,
	so in fact $\Tr  (-1)^Q Q_2 e^{-\Delta(t_1, t_2)}$
	vanishes for all $t_1, t_2>0$.
\end{proof}

Thus the interesting case (at least under the orientation-preserving assumption)
is when both dimensions $n_1$ and $n_2$ are odd.
The metric independence of analytic torsion, Corollary \ref{coro: n odd, acyclic}, generalizes to the following,
which is our main theorem.

\begin{theorem} \label{theorem: multi torsion metric independence}
	For $j = 1, 2$,
	if $n_j$ is odd and if for every $\gamma_j \in \Gamma_j $,
	$\gamma_j$ is either orientation-preserving or has nondegenerate fixed points as a diffeomorphism of $M_j$,
	then $MT(h_1, h_2)$ is independent of the metric $h_j$.
\end{theorem}

We will give the proof in \S \ref{subsection: proof of multi independence}.
The hypotheses ensure the vanishing of  boundary terms in our Stokes' theorem argument
involving the closed two-form $\omega_{MT}$ (to be introduced in \S \ref{subsection: the multi form}).


\section{Heat kernel multi-asymptotics}
\label{sec: Heat kernel multi-asymptotics}

\subsection{Multi-admissibility of the heat trace}
\label{subsection: proof of heat kernel multi-asymptotics}
We will now prove Theorem \ref{theorem: heat kernel is multi-admissible},
which we restate for convenience:
\begin{theorem*}[Theorem \ref{theorem: heat kernel is multi-admissible}]
	Suppose both $F_1 \to M_1$ and $F_2 \to M_2$ are acyclic.
	For $j=1,2$, let $\alpha_j \in C^\infty(M_j,  \End \Lambda^\bullet F_j )$, and suppose that $\alpha_j$ commutes with the action of  $\Gamma_j$,
	inducing an operator $\alpha = \alpha_1 \alpha_2 := \alpha_1 \otimes \alpha_2 $ that commutes with the action of $\Gamma$ and descends to an operator $\alpha_\Gamma \in C^\infty(M_\Gamma, \End  E_\Gamma)$.
	Then $\Tr \alpha_\Gamma e^{-\Delta(t_1, t_2)} $ is multi-admissible.
\end{theorem*}

\begin{proof}
	Let $\Pi: M_1 \times M_2 \to M_\Gamma$ denote the projection,
	which extends to a map from $F$ to $F_\Gamma$
	that is locally a  flat vector bundle isometry.
	For $x \in M_\Gamma$ and $\tilde{x} \in \Pi^{-1}(x) \subset M_1 \times M_2$, let $\Pi_{\tilde{x}}^\ast: (\Lambda^\bullet F_\Gamma)_{x} \to (\Lambda^\bullet F)_{\tilde{x}}$
	be the induced identification of fibers.	
	By the Selberg principle \cite{selberg},
	we have the following relationship between heat kernels:
	\begin{equation} \label{equation: sum on gamma}
	k^{M_\Gamma}(t_1, t_2; x,y)  = \sum_{\gamma \in \Gamma} \left( \Pi^\ast_{\tilde{x}} \right)^{-1}   \gamma_{\tilde{x}}^\ast  
	k^{M_1 \times M_2}(t_1, t_2; \gamma(\tilde{x}), \tilde{y})   \Pi^\ast_{\tilde{y}},
	\end{equation}
	where $\tilde{x} = (\tilde{x}_1, \tilde{x}_2), \tilde{y} = (\tilde{y}_1, \tilde{y}_2) \in M_1 \times M_2$
	are any lifts of $x,y \in M_\Gamma$.
	The heat kernel $k^{M_1 \times M_2}(t_1, t_2; \tilde{x}, \tilde{y})$ on $F$-valued forms
	is the tensor product of the factor heat kernels in the following sense:
	\begin{equation} \label{kernel tensor}
	k^{M_1 \times M_2}(t_1, t_2; \tilde{x}, \tilde{y}) =  k^{M_1}(t_1; \tilde{x}_1, \tilde{y}_1) \otimes k^{M_2}(t_2; \tilde{x}_2, \tilde{y}_2),
	\end{equation}
	where for $j=1,2$, $k^{M_j}(t; \tilde{x}_j, \tilde{y}_j) \in \Hom( (\Lambda^\bullet F_{j})_{\tilde{y}_j}, ( \Lambda^\bullet F_{j})_{\tilde{x}_j}) $
	denotes the heat kernel on $F_j$-valued forms on $M_j$.
	Implicit in the notation is the identification of fibers
	$(\Lambda^\bullet F)_{\tilde{z}} \cong (\Lambda^\bullet F_{1})_{\tilde{z}_1} \otimes (\Lambda^\bullet F_{2})_{\tilde{z}_2} $.
	
	From \eqref{equation: sum on gamma}
	and \eqref{kernel tensor} , we obtain
	\begin{multline*}  
	k^{M_\Gamma}(t_1, t_2; x,y) \\
	=  \sum_{\gamma = (\gamma_1, \gamma_2) \in \Gamma}
	\left( \Pi^\ast_{\tilde{x}} \right)^{-1}   \left(
	(\gamma_1)_{\tilde{x}_1}^\ast k^{M_1}(t_1; \gamma_1(\tilde{x}_1), \tilde{y}_1 ) \otimes
	(\gamma_2)_{\tilde{x}_2}^\ast k^{M_2}(t_2; \gamma_2(\tilde{x}_2), \tilde{y}_2 )  \right)
	\Pi^\ast_{\tilde{y}} .
	\end{multline*}
	Now, at a point $(x,x)$ on the diagonal of $M_\Gamma \times M_\Gamma$, we have
	\begin{align*}
	& k^{M_\Gamma}(t_1, t_2; x,x) \nonumber \\
	&=  \sum_{\gamma = (\gamma_1, \gamma_2) \in \Gamma}  
	\left( \Pi^\ast_{\tilde{x}} \right)^{-1}   \left(
	(\gamma_1)_{\tilde{x}_1}^\ast k^{M_1}(t_1; \gamma_1(\tilde{x}_1), \tilde{x}_1 ) \otimes
	(\gamma_2)_{\tilde{x}_2}^\ast k^{M_2}(t_2; \gamma_2(\tilde{x}_2), \tilde{x}_2 )  \right)
	\Pi^\ast_{\tilde{x}} . 
	\end{align*}
	Let $\alpha = \alpha_1 \otimes \alpha_2$ and $\alpha_\Gamma$ be as in the hypotheses
	of the theorem.
	Then the integral kernel of $\alpha_\Gamma e^{-\Delta(t_1, t_2)}$ has the value at $(x,x)$
	\begin{align*}
	& (\alpha_\Gamma)_x k^{M_\Gamma}(t_1, t_2; x,x) \nonumber \\
	&=  (\alpha_\Gamma)_x \sum_{\gamma = (\gamma_1, \gamma_2) \in \Gamma}  
	\left( \Pi^\ast_{\tilde{x}} \right)^{-1}   \left(
	(\gamma_1)_{\tilde{x}_1}^\ast k^{M_1}(t_1; \gamma_1(\tilde{x}_1), \tilde{x}_1 ) 
	\otimes
	(\gamma_2)_{\tilde{x}_2}^\ast k^{M_2}(t_2; \gamma_2(\tilde{x}_2), \tilde{x}_2 )  \right)
	\Pi^\ast_{\tilde{x}}  \\
	&=  \sum_{\gamma = (\gamma_1, \gamma_2) \in \Gamma}  
	\left( \Pi^\ast_{\tilde{x}} \right)^{-1}   \left(
	\left( \alpha_1 \right)_{\tilde{x}_1}  (\gamma_1)_{\tilde{x}_1}^\ast k^{M_1}(t_1; \gamma_1(\tilde{x}_1), \tilde{x}_1 )
	\otimes \right. 
	\left. \left( \alpha_2 \right)_{\tilde{x}_2}  (\gamma_2)_{\tilde{x}_2}^\ast k^{M_2}(t_2; \gamma_2(\tilde{x}_2), \tilde{x}_2 )  \right)
	\Pi^\ast_{\tilde{x}} 
	\end{align*}
	since $(\alpha_\Gamma)_x \left( \Pi^\ast_{\tilde{x}} \right)^{-1} = \left( \Pi^\ast_{\tilde{x}} \right)^{-1} \alpha_{\tilde{x}}$
	and $\alpha= \alpha_1 \otimes \alpha_2$.
	
	We now take the trace on the fiber at $x$. By the cyclicity of the trace
	and the multiplicativity of the trace of a tensor product, we obtain
	\begin{multline*}
	\tr (\alpha_\Gamma)_x k^{M_\Gamma}(t_1, t_2; x,x)
	=  \sum_{\gamma = (\gamma_1, \gamma_2) \in \Gamma}  
	\left(
	\tr \left( \alpha_1 \right)_{\tilde{x}_1}  (\gamma_1)_{\tilde{x}_1}^\ast k^{M_1}(t_1; \gamma_1(\tilde{x}_1), \tilde{x}_1 ) \right) \\
	\left( \tr \left( \alpha_2 \right)_{\tilde{x}_2}  (\gamma_2)_{\tilde{x}_2}^\ast k^{M_2}(t_2; \gamma_2(\tilde{x}_2), \tilde{x}_2 )  \right). 
	\end{multline*}
	For the $L^2$-trace, we integrate over $M_\Gamma$ to obtain
	\begin{align}
	\Tr \alpha_\Gamma e^{-\Delta(t_1, t_2)} &=
	\int_{M_\Gamma} \tr (\alpha_\Gamma)_x k^{M_\Gamma}(t_1, t_2; x,x)  \, \vol_{M_\Gamma} (x) \nonumber \\ 
	&= \sum_{\gamma \in \Gamma } h_\gamma(t_1, t_2), \label{eq: sum over gamma}
	\end{align}
	where $h_\gamma(t_1, t_2)$ is defined by
	\begin{multline*}
	h_\gamma(t_1, t_2)  := 
	\int_{M_\Gamma} \left(  \tr \left( \alpha_1 \right)_{\tilde{x}_1}  (\gamma_1)_{\tilde{x}_1}^\ast k^{M_1}(t_1; \gamma_1(\tilde{x}_1), \tilde{x}_1 ) \right) \\
	\left( \tr \left( \alpha_2 \right)_{\tilde{x}_2}  (\gamma_2)_{\tilde{x}_2}^\ast k^{M_2}(t_2; \gamma_2(\tilde{x}_2), \tilde{x}_2 )  \right) \vol_{M_\Gamma} (x).
	\end{multline*}
	Let $D \subset M_1 \times M_2$ be a fundamental domain for the cover $M_1 \times M_2 \to M_\Gamma$.
	Then we have 
	\begin{multline*}
	h_\gamma(t_1, t_2) = 
	\int_{D} \left( \tr \left( \alpha_1 \right)_{\tilde{x}_1}  (\gamma_1)_{\tilde{x}_1}^\ast k^{M_1}(t_1; \gamma_1(\tilde{x}_1), \tilde{x}_1 ) \right)  \\
	\left( \tr \left( \alpha_2 \right)_{\tilde{x}_2}  (\gamma_2)_{\tilde{x}_2}^\ast k^{M_2}(t_2; \gamma_2(\tilde{x}_2), \tilde{x}_2 )  \right) \vol_{M_1 \times M_2} (\tilde{x}).
	\end{multline*}
	But by the $\Gamma$-invariance of all the relevant operators,
	this is the same (up to the size of $\Gamma$, which is the number of sheets in the cover) as integrating over all of $M_1 \times M_2$,
	which we can decompose as a product:
	\begin{align*}
	h_\gamma(t_1, t_2)  =
	& \frac{1}{|\Gamma|} \int_{M_1 \times M_2}    \left( \tr \left( \alpha_1 \right)_{\tilde{x}_1}  (\gamma_1)_{\tilde{x}_1}^\ast k^{M_1}(t_1; \gamma_1(\tilde{x}_1), \tilde{x}_1 )  \right)  \\
	& \hspace{57pt}     \left(  \tr \left( \alpha_2 \right)_{\tilde{x}_2}  (\gamma_2)_{\tilde{x}_2}^\ast k^{M_2}(t_2; \gamma_2(\tilde{x}_2), \tilde{x}_2 )  \right) \, \vol_{M_1 \times M_2}(\tilde{x}) \\
	=&      \frac{1}{|\Gamma|}  \left( \int_{M_1}  \left( \tr \left( \alpha_1 \right)_{\tilde{x}_1}  (\gamma_1)_{\tilde{x}_1}^\ast k^{M_1}(t_1; \gamma_1(\tilde{x}_1), \tilde{x}_1 ) \right) \vol_{M_1}(\tilde{x}_1) \right) \\
	& \hspace{17pt}   \left( \int_{M_2}  \left(  \tr \left( \alpha_2 \right)_{\tilde{x}_2}  (\gamma_2)_{\tilde{x}_2}^\ast k^{M_2}(t_2; \gamma_2(\tilde{x}_2), \tilde{x}_2) \right) \vol_{M_2}(\tilde{x}_2)  \right).
	\end{align*}
	We claim that for $j=1,2$, the following is an admissible function of $t_j$
	in the sense of Definition \ref{defn: admissible}:
	\begin{equation}  \label{equation: gamma, j factor}
	\int_{M_j} \left(
	\tr \left( \alpha_j \right)_{\tilde{x}_j}  (\gamma_j)_{\tilde{x}_j}^\ast k^{M_1}(t_1; \gamma_j(\tilde{x}_j), \tilde{x}_j ) \right) \vol_{M_j}(\tilde{x}_j).
	\end{equation}
	To see this, note
	that the expression \eqref{equation: gamma, j factor} is precisely the trace of the operator $\alpha_j \gamma^\ast_j e^{-t_j \Delta_j}$ on $L^2(M_j, \Lambda^\bullet F_j)$.
	The $t_j \to 0$ asymptotic expansion of  $\Tr \alpha_j \gamma^\ast_j e^{-t_j \Delta_j}$
	follows from a well-known fact
	(see, e.g., the book of Gilkey \cite{gilkey}),
	and the decay as $t_j \to \infty$ follows from the
	assumption that $F_j$ is acyclic, i.e., that $\Delta_j$ has trivial kernel.
	
	As a consequence, for every $\gamma \in \Gamma$, $h_\gamma(t_1, t_2)$
	is multi-admissible, being a product
	of an admissible function of $t_1$ and an
	admissible function of $t_2$. Summing over $\gamma \in \Gamma$,
	as in \eqref{eq: sum over gamma},
	shows that $\Tr \alpha_\Gamma e^{-\Delta(t_1, t_2)}$ is multi-admissible,
	completing the proof.	
\end{proof}

\subsection{Vanishing of constant terms}
\label{subsection: Vanishing of constant terms}

To study constant terms in certain (multi-)asymptotic expansions,
we first will review some general results,
for which we will introduce the following new notation.
Consider a compact oriented Riemannian manifold $N$ and a flat orthogonal
or unitary vector bundle $V \to N$.
Let $\phi$ be a Riemannian isometry of $N$ that extends to a flat isometry
of $V$.
Let $\phi^\ast$ denote the induced pullback operator on sections of $\Lambda^\bullet V$.
Let $\Delta$ be the Laplacian on sections of $\Lambda^\bullet V$.
Let $\sigma \in C^\infty(N, \End(\Lambda^\bullet V))$.
Retain these assumptions for Lemmas \ref{lemma: codimension of fixed point},
\ref{lemma: orientation-preserving}, and
\ref{lemma: phi t^0 term} and Corollary \ref{corollary: t^0 vanishing}.

\begin{lemma} \label{lemma: codimension of fixed point}
	Let $X_0$ be a submanifold of $N$ that is a connected component of the fixed point set of $\phi: N \to N$.
	Let $c$ be the codimension of $X_0$ in $N$.
	Then   $\phi$ is orientation-preserving
	as a diffeomorphism of $N$ if and only if $c$ is even, and equivalently,
	$\phi$ is orientation-reversing if and only if $c$ is odd.
\end{lemma}

\begin{lemma} \label{lemma: orientation-preserving}
	Suppose $\phi$ is orientation-preserving and $\dim N$ is odd.
	Then there is no $t^0$ term in the asymptotic expansion of
	$\Tr \sigma  \phi^\ast  e^{-t \Delta}$.
\end{lemma}
For a proof of Lemma \ref{lemma: orientation-preserving}, see the proof of Proposition 2 in \cite{lott-rothenberg}.

For a proof of the following result, see, e.g., the book of Gilkey \cite{gilkey}.
It is originally due to Kotake \cite{kotake},
who was the first to apply heat equation methods to
Atiyah-Bott's Lefschetz fixed point theorem \cite{atiyah-bott}.

\begin{lemma} \label{lemma: phi t^0 term}
	Suppose that $\phi$ has nondegenerate fixed points.
	Then $\Tr \sigma  \phi^\ast e^{-t \Delta}$ is bounded,
	and furthermore, as $t \to 0^+$,
	\begin{equation*}
	\Tr \sigma \phi^\ast e^{-t\Delta} = \sum_{ x_0 \in \operatorname{Fix}(\phi)} \frac{\tr \sigma_{x_0} \phi^\ast_{x_0}}{| \det(I - d\phi_{x_0})| } 
	\,+\, O\left( t^{\frac{1}{2}} \right).
	\end{equation*}
	where  $\operatorname{Fix}(\phi) \subset N$ denotes the finite set of fixed points of $\phi$.
\end{lemma}

We will need the following vanishing result in the case in which the operator $\sigma$
takes a certain special form:

\begin{coro} \label{corollary: t^0 vanishing}
	Suppose that $\phi$ has nondegenerate fixed points.
	Let $g(u)$ be a smooth one-parameter family of $\phi$-invariant Riemannian metrics on $N$,
	with associated Hodge star $\ast = \ast(u)$.
	Then
	$\Tr (-1)^Q \ast^{-1} \frac{d \ast}{du}  \phi^\ast e^{-t\Delta^{g(u)}}$ is $O(t^{1/2})$ as $t \to 0$.
\end{coro}
\begin{proof}
	We will show that for each fixed point $x_0$,
	\begin{equation} \label{equation: suffices to show vanishing} 
	\tr (-1)^Q \ast_{x_0}^{-1} \frac{d \ast_{x_0}}{du}  \phi^\ast_{x_0} = 0 ,
	\end{equation}
	which suffices to prove the claim by Lemma \ref{lemma: phi t^0 term}.
	Our approach is to conjugate the operator on the left-hand side of \eqref{equation: suffices to show vanishing}
	by $\ast_{x_0}$.
	Note that $\ast (-1)^Q = (-1)^{\dim N}(-1)^Q \ast$,
	$\ast \frac{d \ast}{du} = - \frac{d \ast}{du} \ast$ (since $\ast$ squares to a constant),
	and $\ast \phi^\ast  = \epsilon_{\text{or.}} \phi^\ast \ast$ where $\epsilon_{\text{or.}} = 1$ if $\phi$ is orientation-preserving 
	and $\epsilon_{\text{or.}}=-1$ if $\phi$
	is orientation-reversing.
	But since $\phi$ has nondegenerate fixed points, $\phi$ is orientation-preserving if and only if $\dim N$ is even,
	which means $\epsilon_{\text{or.}} = (-1)^{\dim N}$.
	The cyclicity of the trace and the facts above show that
	\begin{align*}
	\tr (-1)^Q \ast_{x_0}^{-1} \frac{d \ast_{x_0}}{du}  \phi^\ast_{x_0} &= \tr \ast_{x_0} (-1)^Q \ast_{x_0}^{-1} \frac{d \ast_{x_0}}{du}  \phi^\ast_{x_0} \ast_{x_0}^{-1} \\
	&= - \tr (-1)^Q \ast_{x_0}^{-1} \frac{d \ast_{x_0}}{du}  \phi^\ast_{x_0},
	\end{align*}
	which proves \eqref{equation: suffices to show vanishing} and completes the proof.
\end{proof}

Now we return to the setting of Theorem \ref{theorem: heat kernel is multi-admissible}.
The following results
will be essential to our proof of the metric independence theorem
for multi-torsion.

\begin{prop} \label{prop: no t^0 terms}
	Retain the assumptions of Theorem \ref{theorem: heat kernel is multi-admissible}.
	Suppose the dimension $n_1$ is odd and
	for every $\gamma_1 \in \Gamma_1 $,
	$\gamma_1$ is either orientation-preserving or has nondegenerate fixed points as a diffeomorphism of $M_1$.
	Suppose $\alpha_1$ is of the special form $\alpha_1 = (-1)^{Q_1} \ast_1^{-1} \frac{d \ast_1}{du}$
	for a smooth one-parameter family of metrics $h_1(u)$ on $M_1$.
	Then there are no $(t_1)^0$ terms in the multi-asymptotic expansions of \ref{first}-\ref{third}
	for the multi-admissible function
	$\Tr  \alpha_\Gamma e^{-\Delta(t_1, t_2)}$.
	A similar statement holds with the roles of $1$ and $2$ reversed.
\end{prop}
\begin{proof}
	By \eqref{eq: sum over gamma} and the discussion following it,
	it suffices to prove the claim that for each $\gamma_1 \in \Gamma_1$,
	$\Tr (-1)^{Q_1} \alpha_1 \gamma^\ast_1 e^{-t_1 \Delta_1}$
	has no $(t_1)^0$ term in its small $t_1$ asymptotic expansion. There are two cases:
	
	First, if $\gamma_1$ is orientation-preserving, then the claim follows from Lemma \ref{lemma: orientation-preserving} since $n_1$ is odd.
	
	Second, if $\gamma_1$ has nondegenerate fixed points, then the claim follows from Corollary \ref{corollary: t^0 vanishing}.
\end{proof}

\begin{coro} \label{coro: partial derivative vanish}
	Under the assumptions of Proposition \ref{prop: no t^0 terms}, we have
	\begin{equation*}
	\left. \frac{\partial}{\partial s_2} \right|_{(s_1, s_2)=(0,0)}   \zeta \left( s_1, s_2; -\Delta(t_1, t_2, u),  \alpha_\Gamma \right) = 0.
	\end{equation*}
	A similar statement holds with the roles of $1$ and $2$ reversed.
\end{coro}
\begin{proof}
	This follows from Proposition \ref{prop: no t^0 terms} and Lemma \ref{lemma: multi-zeta properties}.
\end{proof}


\section{Metric independence theorem for multi-torsion}

\label{sec: Metric independence theorem for multi-torsion}

\subsection{A closed form for multi-torsion}
\label{subsection: the multi form}

We will now introduce the closed form $\omega_{MT}$
and interpret multi-torsion as the integral of $\omega_{MT}$ over a surface in the space of metrics.

The form $\omega_{MT}$ makes sense in a somewhat more general setting than
that in which we defined multi-torsion.
What is essential is that there is a sufficiently nice decomposition of $d$, $d^\ast$, $\Delta$, and $\delta^\M$.
More precisely,
we make the following assumptions.

Let $M$ be a compact manifold of dimension $n$.
Let $(E, d)$ be an elliptic complex on $M$, and
let $(-1)^Q$ denote the grading operator.
We assume that $d$ admits a decomposition
\begin{equation*}
d = d_1 + \dots + d_m
\end{equation*}
such that for all $i$ and $j$, $d_i$ and $d_j$ anticommute
and $d_i$ anticommutes with $(-1)^Q$.
(Here, the subscript $j$ on $d_j$ does \textit{not}
refer to the $\mathbb{Z}$-grading on $E$.)
This implies that for any metric on $E$, we have a decomposition of the associated $d^\ast$:
\begin{equation*}
d^\ast = d^\ast_1 + \dots + d^\ast_m.
\end{equation*}

Since we are allowing the density to vary in this section,
we will now consider a ``metric" on $E$
to be a choice of both a metric in the usual sense and a density,
i.e., a metric is a section of $E^\ast \otimes E^\ast \otimes |\Lambda|$,
where $|\Lambda|$ denotes the density bundle,
a trivial line bundle on $M$.
We may view a metric $h$ as an isomorphism $h: E \to E^\ast \otimes |\Lambda|$.

We will assume that $\M$ is a space of metrics on $E$
such that for any metric in $\M$, the associated $d^\ast_i$'s satisfy that
for all $i$ and $j$, $d^\ast_i$ and $d^\star_j$ anticommute
and that if $i \neq j$,
$d^\star_i$ and $d_j$ anticommute.
Then the Laplacian decomposes as $\Delta = \sum_{k=1}^m \Delta_k$,
where $\Delta_k := d_k d_k^\ast + d_k^\ast d_k$;
$d_k$ and $d_k^\ast$ commute with $\Delta$ and therefore with functions of $\Delta$,
in particular  with the resolvent $R = R_z := (z - \Delta)^{-1}$.

Furthermore, we will require that the $\M$-exterior derivative $\delta^\M$ decomposes as
\begin{equation*}
\delta^\M = \delta^\M_1 + \cdots + \delta^\M_m
\end{equation*}
such that $\delta^\M_i d^\ast_j = 0$ for $i \neq j$.
Note that the one-form $b = h^{-1} \delta^\M h$
introduced in \S \ref{subsection: the one-form}
still makes sense.
Our assumptions imply that $b$ decomposes as
\begin{equation*}
b = b_1 + \dots + b_m,
\end{equation*}
where $b_j := h^{-1} \delta^\M_j h$.
Lemma \ref{lemma: derivative of d star and Delta} generalizes
to give that $\delta^\M d^\ast_j = \delta^\M_j d^\ast_j = [d^\ast_j, b_j]$.
We will assume that $b_j$ commutes with $d_i$ and $d^\ast_i$ for $i \neq j$,
and that each $b_j$ is symmetric.
Thus we assume that a generalization of Lemma \ref{lemma: b is symmetric} holds.
In the geometric setting of \S \ref{subsection: The geometry of finite quotients of product manifolds},
this fact holds automatically for the same reason that Lemma \ref{lemma: b is symmetric} holds.

\begin{remark} \label{remark: hodge star}
	The geometric setting of \S \ref{subsection: The geometry of finite quotients of product manifolds}
	satisfies the assumptions above, with $m = 2$ and with $\M = \Mprod$.
	In that setting,
	we may rewrite $b$ using the Hodge star,
	as we will now explain.
	For Riemannian metrics $h_1$ on $M_1$ and $h_2$ on $M_2$,
	let $\ast_1$ and $\ast_2$ be the associated Hodge stars on $\Lambda^\bullet F_1$
	and $\Lambda^\bullet F_2$, respectively.
	If $\Gamma_1$ and $\Gamma_2$ act by orientation-preserving diffeomorphisms,
	then $\ast_1$ and $\ast_2$ descend to well-defined operators on $\Lambda^\bullet F_\Gamma$,
	but even if not,
	the $C^\infty(M_\Gamma, \End(\Lambda^\bullet F_\Gamma))$-valued
	one-forms $\ast_1^{-1} (\delta^\M_1 \ast_1)$ and $\ast_2^{-1} (\delta^\M_1 \ast_2)$
	are well-defined (since $\ast_1$ and $\ast_2$
	are well-defined up to a sign),
	and we may write $b = b_1 + b_2$, where
	$b_1 = \ast_1^{-1} (\delta^\M_1 \ast_1)$ and $b_2 = \ast_2^{-1} (\delta^\M_1 \ast_2)$.
\end{remark}

Defining $\omega_{MT}$ and $\tilde{\omega}_{MT}$,
which are $m$-forms that generalize the one-forms $\omega_T$ and $\tilde{\omega}_T$, respectively,
will require  introducing some notation.
Fix an integer $N$ that is sufficiently
large to ensure  that $R_z^{N}$ is trace-class.
Let $\C_{N,m}$ denote the finite set of compositions (i.e., ordered partitions) of the positive integer $N$ into $m$ positive integer summands.
We will view $\C_{N,m}$ as the set of multi-indices $I = (i_1, \dots, i_m)$
such that each $i_j$ is a strictly positive integer and $N = i_1 + \cdots + i_m$.
We will view $S_{m-1}$ as the permutation group of the set $\{2, \dots, m \}$.
For $\sigma \in S_{m-1}$, $I  = (i_1, \dots, i_m) \in \C_{N,m}$, and $z$ a complex number not in the spectrum of $\Delta$,
let $\R_{z, \sigma, I}$ be defined by
\begin{align*}
\R_{z, \sigma, I}  :=  R_z^{i_1} b_{1} \, R_z^{i_2} b_{\sigma(2)} \dotsm R_z^{i_m} b_{\sigma(m)}.
\end{align*}
Let $T_{\sigma, I}$ and $\conj{T}_{\sigma, I}$ be defined by
\begin{align*}
T_{\sigma, I} &:= \sign \sigma ~ \Tr (-1)^Q \R_{z,\sigma, I} \text{ and} \\
\conj{T}_{ \sigma, I} &:= \sign \sigma ~ \Trconj (-1)^Q  \R_{\conj{z},\sigma, I}.
\end{align*}
$T_{\sigma, I}$ and $\conj{T}_{\sigma, I}$ depend on $z$, but we suppress this.
Note that $T_{\sigma, I}$ is {\em not} necessarily equal to the complex conjugate of $\conj{T}_{\sigma, I}$.

Our $m$-form $\tilde{\omega}_{MT}$ is defined by
\begin{equation}  \label{eq: defn of tilde omega MT}
\tilde{\omega}_{MT} := \frac{1}{ |\C_{N,m}| (m-1)!} \sum_{\sigma \in S_{m-1}, I \in \C_{N, m}} \frac{1}{2} \left( T_{\sigma,I} + \conj{T}_{\sigma,I} \right).
\end{equation}
The normalizing constant $|\C_{N,m}| (m-1)!$ is the number of summands in the sum.
We have symmetrized over all orderings of $2, \dots, m$;
we could have instead chosen to symmetrize over all orderings of $1, \dots, m$,
which would have been equivalent by the graded cyclicity of the trace.
Again, we suppress the dependence on $z$ (and on $N$).

\begin{remark} \label{remark: m=1}
 Consider the special case when $m=1$. Then the sum in \eqref{eq: defn of tilde omega MT}
 is trivial, and we have
 \begin{equation*} 
 \tilde{\omega}_{MT} = \frac{1}{2} \left(  \Tr (-1)^Q  R_z^N b + \Trconj (-1)^Q R_{\conj{z}}^N b \right),
 \end{equation*}
 where we have set $b := b_1$.
 An application of the trace-adjoint identity (Lemma \ref{lemma: adjoint}) shows that
 this expression equals the torsion one-form
 $\tilde{\omega}_T$ as defined in \eqref{eq: def of omega T tilde}.
 Thus our results for $\tilde{\omega}_{MT}$
 apply in particular to $\tilde{\omega}_T$.
\end{remark}

We define $\omega_{MT}$ in terms of $\tilde{\omega}_{MT}$ via the  contour integral
\begin{equation} \label{eq: defn of omega MT}
\omega_{MT} := \frac{1}{(N-1)!} \frac{1}{2\pi i} \int_C  e^{-z}\,  \tilde{\omega}_{MT} \, dz ,
\end{equation}
where $C$ is a contour in $\mathbb{C}$ enclosing $[0,\infty)$.
\begin{remark}
	There is not an obvious simpler formula for $\omega_{MT}$ involving the heat operator
	(analogous to \eqref{eq: def of omega T})
	because $b_1, \dots, b_m$ do not commute with the resolvent $R_z$ in general.
	But in the most important special case,
	the contour integral in \eqref{eq: defn of omega MT}
	is simple to compute, giving a heat operator;
	see \eqref{equation: omega MT on sigma} below.
\end{remark}

We will now explain the significance of $\omega_{MT}$ to the multi-torsion.
Let us return to the setting of \S \ref{subsection: The geometry of finite quotients of product manifolds}, in which we have a bigrading of $E$: $E = \bigoplus_{q_1, q_2} E^{q_1, q_2}$.
Let $h_1 \times h_2 $ be a product Riemannian metric on $M_1 \times M_2$.
Let $h^{q_1, q_2}$ be the induced metric on $E^{q_1, q_2}$ and let $\Delta^h = \Delta_1^h + \Delta_2^h$ be the associated Laplacian.
For $t_1, t_2 > 0$, consider the product Riemannian metric
$\left( \frac{1}{t_1} h_1 \right) \times \left( \frac{1}{t_2} h_2 \right)$,
which induces a metric $h_{t_1, t_2}$ on $E$ whose restriction to
$E^{q_1, q_2}$ is $t_1^{q_1 - n_1/2} t_2^{q_2 - n_2/2} h^{q_1, q_2}$.
(Note that the factors of $t_1^{- n_1/2}$ and  $t_2^{- n_2/2}$
are from the dependence of the respective Riemannian densities on $t_1$ and $t_2$,
which we did not consider in Remark \ref{remark: scale Riemannian metric}
since we had fixed a density then.)
\label{h_t_1_t_2}
Then for $j=1, 2$, we have, on the surface $\Sigma_{h_1, h_2}$ in $\Mprod$ parametrized by $(t_1, t_2) \mapsto h_{t_1, t_2}$:
\begin{align*}
b_j = \left( q_j - \frac{n_j}{2} \right) \frac{dt_j}{t_j};  ~
d^\ast_j(t_j) = t_j d_j^\ast(h); ~
\Delta(t_1, t_2) = t_1 \Delta_1^h + t_2 \Delta_2^h.
\end{align*}
Pulled back to $\Sigma_{h_1, h_2}$, every summand in the sum in \eqref{eq: defn of tilde omega MT}
is the same; i.e., we have for every $\sigma, I$ that
\begin{equation*}
\frac{1}{2} \left( T_{\sigma,I} + \conj{T}_{\sigma,I} \right) = \Tr (-1)^Q \left(Q_1 - \frac{n_1}{2} I \right) \left(Q_2 - \frac{n_2}{2} I \right) (t_1 \Delta_1^h + t_2 \Delta_2^h - z)^{-N} \, \frac{dt_1}{t_1} \frac{dt_2}{t_2},
\end{equation*}
which follows from the graded cyclicity of the trace, the adjoint-trace relation of Lemma \ref{lemma: adjoint}, and the fact that $Q_1$ and $Q_2$ commute with the resolvent $R_z$.
Thus the $2$-form $\tilde{\omega}_{MT}$ pulled back to $\Sigma_{h_1, h_2}$ is
\begin{equation*}
\tilde{\omega}_{MT} = \Tr (-1)^Q   \left(Q_1 - \frac{n_1}{2} I \right) \left(Q_2 - \frac{n_2}{2} I \right)  (t_1 \Delta_1^h + t_2 \Delta_2^h - z)^{-N} \,  \frac{dt_1}{t_1} \frac{dt_2}{t_2}.
\end{equation*}
Performing the contour integral gives in turn that $\omega_{MT}$ pulled back to $\Sigma_{h_1, h_2}$ is
\begin{equation*}
\omega_{MT} = \Tr (-1)^Q   \left(Q_1 - \frac{n_1}{2} I \right) \left(Q_2 - \frac{n_2}{2} I \right) e^{-(t_1 \Delta_1^h + t_2 \Delta_2^h)}  \frac{dt_1}{t_1} \frac{dt_2}{t_2} .
\end{equation*}
Expanding $ \left(Q_1 - \frac{n_1}{2} I \right) \left(Q_2 - \frac{n_2}{2} I \right)$
gives four terms, three of which
vanish by the argument in the proof of Theorem
\ref{theorem: even vanishing of MT}
or by a similar argument,
so $\omega_{MT}$ simplifies to
\begin{equation} \label{equation: omega MT on sigma}
\omega_{MT} = \Tr (-1)^Q  Q_1 Q_2 e^{-(t_1 \Delta_1^h + t_2 \Delta_2^h)}  \frac{dt_1}{t_1} \frac{dt_2}{t_2} .
\end{equation}
Thus we have proven the following generalization of Lemma \ref{lemma: zeta as integral of omega}:
\begin{lemma} \label{lemma: multi zeta as integral of omega}
	For $\operatorname{Re} s_1, \operatorname{Re} s_2 $ large,
	\begin{equation*}
	\zeta(s_1, s_2;  \Delta^h(t_1, t_2), (-1)^Q Q_1 Q_2) =  \frac{1}{\Gamma(s_1)\Gamma(s_2)} \int_{\Sigma_{h_1, h_2}} t_1^{s_1} t_2^{s_2} \omega_{MT}.
	\end{equation*}
\end{lemma}
Using that $\frac{1}{\Gamma(s)} = s + O(s^2)$ as $s \to 0$ and Lemma \ref{lemma: multi zeta as integral of omega},
and recalling Definition \ref{definition: multi-torsion} of the multi-torsion $MT(h_1, h_2)$,
we have shown that
\begin{equation*}
4\, MT(h_1, h_2) =  \left. \left( \int_{\Sigma_{h_1, h_2}} t_1^{s_1} t_2^{s_2} \omega_{MT} \right) \right|^{\text{AC}}_{(s_1, s_2)=(0,0)},
\end{equation*}
where the superscript $\text{AC}$ indicates that we must
analytically continue the function in parentheses to the origin.
Formally, setting $s_1 = s_2=0$ on the right-hand side leaves $\int_{\Sigma_{h_1, h_2}} \omega_{MT}$;
this is purely formal because the regularization is necessary for the integral to converge.
This suggests the following heuristic interpretation of multi-torsion, generalizing our interpretation of analytic torsion:
\begin{displayquote}
	The multi-torsion $MT(h_1, h_2)$ may be interpreted as a regularized integral of $\frac{1}{4} \omega_{MT}$ over the surface $\Sigma_{h_1, h_2}$ in the space of metrics $\Mprod$.
\end{displayquote}

\subsection{\texorpdfstring{Proof that $\omega_{MT}$ is closed}{}}
\label{subsection: Proof that multi form is closed}

We will now prove the following via several lemmas:
\begin{theorem} \label{theorem: multi closed}
	The $m$-forms $\tilde{\omega}_{MT}$ and $\omega_{MT}$ are closed on $\M$.
\end{theorem} 
Since $\delta^\M$ commutes with the integral in \eqref{eq: defn of omega MT},
it suffices to prove that $\delta^\M \tilde{\omega}_{MT} = 0$.
Furthermore, without loss of generality, it suffices to prove that $\delta^\M_1 \tilde{\omega}_{MT} = 0$.

\begin{lemma}
	Consider operator-valued forms $A$, $B$, and $C$.
	Assume that $d_1$ and $d_1^\ast$ commute with each of $A$, $B$, and $C$,
	and that the forms below are trace-class. Then we have the identity
	\begin{equation*}
	\Tr (-1)^Q A b_1 B (\delta^\M_1 \Delta)^\ast C  = \Tr (-1)^Q A (\delta^\M_1 \Delta) B b_1 C.
	\end{equation*}
\end{lemma}
\begin{proof}
	A generalization of Lemma \ref{lemma: derivative of d star and Delta}
	holds, giving 
	$\delta_1^\M \Delta = \{ d_1, [d_1^\ast, b_1] \}$ and therefore
	$(\delta_1^\M \Delta)^\ast = b_1d_1d_1^\ast - d_1 b_1 d_1^\ast + d_1^\ast b_1 d_1 - d_1^\ast d_1 b_1$.
We use the cyclicity of the trace and that $d_1$ and $d_1^\ast$ commute with $A$, $B$, and $C$ and anticommute with $(-1)^Q$ to compute:
  \begin{align*}
   \Tr (-1)^Q A  b_1 B (b_1 d_1 d_1^\ast ) C   &=   \Tr (-1)^Q A  (d_1 d_1^\ast b_1) B b_1  C;  \\
   \Tr (-1)^Q A  b_1 B ( - d_1^\ast  d_1 b_1 ) C   &=  \Tr (-1)^Q A  (-  b_1 d_1^\ast d_1)  B  b_1  C;  \\
   \Tr (-1)^Q A b_1 B (-  d_1 b_1 d_1^\ast + d_1^\ast b_1 d_1 ) C   &=  \Tr  (-1)^Q A (d_1^\ast b_1 d_1 - d_1 b_1 d_1^\ast)  B  b_1  C.
\end{align*}
Summing the three identities above gives the result.
\end{proof}

We will need the following computations
of the derivatives of $T_{\sigma,I}$ and $\conj{T}_{\sigma, I}$ in the $1$-direction:
\begin{align*}
~ \delta^\M_1 (T_{\sigma,I})  = & -   \sign \sigma \Tr (-1)^Q R_z^{i_1} b_1 b_1 R_z^{i_2} b_{\sigma(2)} \dots R_z^{i_m} b_{\sigma(m)}  \\
& + \sum_{\alpha = 1}^{i_1}  \sign \sigma  \Tr (-1)^Q R_z^\alpha (\delta^\M_1 \Delta) R_z^{ i_1 - \alpha + 1 } b_1 R_z^{i_2} b_{\sigma(2)} \dotsm R_z^{i_m} b_{\sigma(m)}   \\
& +  \sum_{j=2}^m \sum_{\beta = 1}^{i_j}  T_{\sigma, I; j, \beta},
\end{align*}
where 
\begin{equation*}
T_{\sigma, I; j, \beta}  := (-1)^{j-1}   \sign \sigma \Tr (-1)^Q R_z^{i_1} b_{1} \dotsm b_{\sigma(k-1)} R_z^{\beta} (\delta^\M_1 \Delta)  R_z^{i_j - \beta + 1} b_{\sigma(j)} \dotsm R_z^{i_m} b_{\sigma(m)}.
\end{equation*}
Similarly,
\begin{align*}
~ \delta^\M_1 (\conj{T}_{\sigma,I})  = & -   \sign \sigma \Trconj (-1)^Q R_{\conj{z}} ^{i_1} b_1 b_1 R_{\conj{z}}^{i_2} b_{\sigma(2)} \dots R_{\conj{z}}^{i_m} b_{\sigma(m)}  \\
& + \sum_{\alpha = 1}^{i_1}  \sign \sigma  \Trconj (-1)^Q R_{\conj{z}}^\alpha (\delta^\M_1 \Delta) R_{\conj{z}}^{ i_1 - \alpha + 1 } b_1 R_{\conj{z}}^{i_2} b_{\sigma(2)} \dotsm R_{\conj{z}}^{i_m} b_{\sigma(m)}   \\
& +  \sum_{j=2}^m \sum_{\beta = 1}^{i_j}  \conj{T}_{\sigma, I; j, \beta},
\end{align*}
where 
\begin{equation*}
\conj{T}_{\sigma, I; j, \beta}  := (-1)^{j-1}   \sign \sigma \Trconj (-1)^Q R_{\conj{z}}^{i_1} b_{1} \dotsm b_{\sigma(k-1)} R_{\conj{z}}^{\beta} (\delta^\M_1 \Delta)  R_{\conj{z}}^{i_j - \beta + 1} b_{\sigma(j)} \dotsm R_{\conj{z}}^{i_m} b_{\sigma(m)}.
\end{equation*}

We will now show $\delta^\M_1 \tilde{\omega}_{MT} = 0$ via three lemmas.

\begin{lemma} The following vanishes:
	\begin{align*}
	&  \sum_{\sigma \in S_{m-1}, I \in \C_{N, m}}  \sign \sigma \Tr (-1)^Q R_z^{i_1} b_1 b_1 R_z^{i_2} b_{\sigma(2)} \dotsm R_z^{i_m} b_{\sigma(m)}  \\
	+&   \sum_{\sigma \in S_{m-1}, I \in \C_{N, m}}  \sign \sigma \Trconj (-1)^Q R_{\conj{z}} ^{i_1} b_1 b_1 R_{\conj{z}}^{i_2} b_{\sigma(2)} \dotsm R_{\conj{z}}^{i_m} b_{\sigma(m)}.
	\end{align*}
\end{lemma}
\begin{proof} We have
	\begin{align*}
	& \sign \sigma\, \Tr (-1)^Q R_z^{i_1} b_1 b_1 R_z^{i_2} b_{\sigma(2)} \dotsm R_z^{i_m} b_{\sigma(m)} \\
	&=  \sign \sigma\, \Trconj  \left( (-1)^Q R_z^{i_1} b_1 b_1 R_z^{i_2} b_{\sigma(2)} \dotsm R_z^{i_m} b_{\sigma(m)} \right)^\ast \\
	&= \sign \sigma\, (-1)^{\frac{1}{2}m(m+1)}   \Trconj  b_{\sigma(m)} R_{\conj{z}}^{i_m}  \dotsm b_{\sigma(2)} R_{\conj{z}}^{i_2}  b_1 b_1 R_{\conj{z}}^{i_1} (-1)^Q \\
	&= \sign \sigma\, (-1)^{\frac{1}{2}m(m+1)}   \Trconj  (-1)^Q  R_{\conj{z}}^{i_2}  b_1 b_1 R_{\conj{z}}^{i_1}  b_{\sigma(m)} R_{\conj{z}}^{i_m}  \dotsm b_{\sigma(2)}.
	\end{align*}
	Let $r$ be the permutation that reverses the order of $(2, \dots, m)$.
	Note that $\sign r = (-1)^{\frac{1}{2} (m-2)(m-1)} = (-1)^{1 + \frac{1}{2} m(m+1)}$,
	so that $\sign ( r \circ \sigma) = (-1)^{1 + \frac{1}{2} m(m+1)} \sign \sigma$.
	Summing over $\sigma \in S_{m-1}$ and $I \in \C_{N,m}$, we obtain the lemma.
\end{proof}

\begin{lemma} The following vanishes:
	\begin{align*}
	&  \sum_{\sigma \in S_{m-1}, I \in \C_{N, m}} \sum_{\alpha = 1}^{i_1}  \sign \sigma  \Tr (-1)^Q R_z^\alpha (\delta^\M_1 \Delta) R_z^{ i_1 - \alpha + 1 } b_1 R_z^{i_2} b_{\sigma(2)} \dotsm R_z^{i_m} b_{\sigma(m)} \\
	+& \sum_{\sigma \in S_{m-1}, I \in \C_{N, m}}  \sum_{\alpha = 1}^{i_1}  \sign \sigma  \Trconj (-1)^Q R_{\conj{z}}^\alpha (\delta^\M_1 \Delta) R_{\conj{z}}^{ i_1 - \alpha + 1 } b_1 R_{\conj{z}}^{i_2} b_{\sigma(2)} \dotsm R_{\conj{z}}^{i_m} b_{\sigma(m)}.
	\end{align*}
\end{lemma}
\begin{proof}
	We have:
	\begin{align*}
	& \sign \sigma  \Tr (-1)^Q R_z^\alpha (\delta^\M_1 \Delta) R_z^{ i_1 - \alpha + 1 } b_1 R_z^{i_2} b_{\sigma(2)} \dotsm R_z^{i_m} b_{\sigma(m)}  \\
	&= \sign \sigma  \Trconj \left( (-1)^Q R_z^\alpha (\delta^\M_1 \Delta) R_z^{ i_1 - \alpha + 1 } b_1 R_z^{i_2} b_{\sigma(2)} \dotsm R_z^{i_m} b_{\sigma(m)}  \right)^\ast \\
	&= \sign \sigma (-1)^{\frac{1}{2}m(m+1)} \Trconj  b_{\sigma(m)}   R_{\conj{z}}^{i_m} \dotsm b_{\sigma(2)}  R_{\conj{z}}^{i_2}  b_1  R_{\conj{z}}^{ i_1 - \alpha + 1 } (\delta^\M_1 \Delta)^\ast R_{\conj{z}}^\alpha (-1)^Q  \\
	&= \sign \sigma (-1)^{\frac{1}{2}m(m+1)} \Trconj  b_{\sigma(m)}   R_{\conj{z}}^{i_m} \dotsm b_{\sigma(2)}  R_{\conj{z}}^{i_2}  (\delta^\M_1 \Delta)  R_{\conj{z}}^{ i_1 - \alpha + 1 } b_1 R_{\conj{z}}^\alpha (-1)^Q  \\
	&= \sign \sigma (-1)^{\frac{1}{2}m(m+1)} \Trconj (-1)^Q  R_{\conj{z}}^{i_2}  (\delta^\M_1 \Delta)  R_{\conj{z}}^{ i_1 - \alpha + 1 } b_1 R_{\conj{z}}^\alpha  b_{\sigma(m)}   R_{\conj{z}}^{i_m} \dotsm b_{\sigma(2)} .
	\end{align*}
	The claim follows from an argument similar to the previous lemma.
\end{proof}

\begin{lemma}
	For each $j$ ($2 \leq j \leq m$),
	\begin{align*}
	\sum_{\sigma \in S_{m-1}, I \in \C_{N, m}}  \sum_{\beta = 1}^{i_j}  \left(   T_{\sigma, I; j, \beta} +   \conj{T}_{\sigma, I; j, \beta}  \right)   = 0.
	\end{align*}
\end{lemma}
\begin{proof}
	We have
	\begin{align*}
	T_{\sigma, I; j, \beta}  = & (-1)^{j-1}   \sign \sigma \\
	&  \Tr (-1)^Q R_z^{i_1} b_{1} \dotsm b_{\sigma(j-1)} R_z^{\beta} (\delta^\M_1 \Delta)  R_z^{i_j - \beta + 1} b_{\sigma(j)} \dotsm R_z^{i_m} b_{\sigma(m)} \\
	= &(-1)^{j-1}    \sign \sigma  \\
	& \Trconj \left( (-1)^Q R_z^{i_1} b_{1} \dotsm b_{\sigma(j-1)} R_z^{\beta} (\delta^\M_1 \Delta)  R_z^{i_j - \beta + 1} b_{\sigma(j)} \dotsm R_z^{i_m} b_{\sigma(m)}  \right)^\ast \\
	= &(-1)^{j-1}  \sign \sigma  (-1)^{\frac{1}{2} m (m+1) }  \\
	& \Trconj   b_{\sigma(m)}  R_{\conj{z}}^{i_m} \dotsm  b_{\sigma(j)} R_{\conj{z}}^{i_j - \beta + 1} (\delta^\M_1 \Delta)^\ast R_{\conj{z}}^{\beta} b_{\sigma(j-1)}  \dotsm b_{1} R_{\conj{z}}^{i_1} (-1)^Q \\
	= &(-1)^{j-1}  \sign \sigma  (-1)^{\frac{1}{2} m (m+1) }  \\
	& \Trconj   b_{\sigma(m)}  R_{\conj{z}}^{i_m} \dotsm  b_{\sigma(j)} R_{\conj{z}}^{i_j - \beta + 1} b_1 R_{\conj{z}}^{\beta} b_{\sigma(j-1)}  \dotsm b_2 R_{\conj{z}}^{i_2} (\delta^\M_1 \Delta) R_{\conj{z}}^{i_1} (-1)^Q \\
	= & (-1)^{j-1}  \sign \sigma  (-1)^{\frac{1}{2} m (m+1)  + j(m-j+1)}  \\
	& \Trconj (-1)^Q  R_{\conj{z}}^{i_j - \beta + 1} b_1 R_{\conj{z}}^{\beta} b_{\sigma(j-1)}  \dotsm  b_2 R_{\conj{z}}^{i_2} (\delta^\M_1 \Delta) R_{\conj{z}}^{i_1}  b_{\sigma(m)}  R_{\conj{z}}^{i_m} \dotsm  b_{\sigma(j)} .
	\end{align*}
	Summing over $\sigma \in S_{m-1}$, $I \in \C_{N,m}$, and $\beta \in \{ 1, \dots, i_j\}$ proves the lemma.
\end{proof}

The previous three lemmas prove that $\delta^\M_1 \tilde{\omega}_{MT} = 0$.
The same argument shows that $\delta^\M_k \tilde{\omega}_{MT} = 0$ for all $k$.
It follows that $\delta^\M \tilde{\omega}_{MT} =0$
and therefore  that $\delta^\M \omega_{MT} =0$. This completes
the proof of Theorem \ref{theorem: multi closed}.

\subsection{Proof of the metric independence theorem}
\label{subsection: proof of multi independence}

We will now prove our main theorem:
\begin{theorem*}[Theorem \ref{theorem: multi torsion metric independence}]
	Suppose that for either  $j = 1$ or $j=2$,
	$n_j$ is odd and for every $\gamma_j \in \Gamma_j $,
	$\gamma_j$ is either orientation-preserving or has nondegenerate fixed points as a diffeomorphism of $M_j$.
	Then $MT(h_1, h_2)$ is independent of the metric $h_j$.
\end{theorem*}
\begin{proof}
	Without loss of generality, we may assume $j=1$. So let us suppose that $n_1$ is odd.
	Fix a metric $h_2$ on $M_2$ and consider a smooth curve $h_1(u)$, $u \in [0,1]$, of metrics on $M_1$, with associated Hodge star $\ast_1 = \ast_1(u)$.
	For $\epsilon_1, \epsilon_2, A_1, A_2 > 0$,
	let $B = B_{\epsilon_1, \epsilon_2, A_1, A_2}$ be the cube in $\Mprod$ parametrized by
	$(t_1, t_2, u) \mapsto \frac{1}{t_1}h_1(u) \times \frac{1}{t_2} h_2$, for $\epsilon_1 \leq t_1 \leq A_1$, $\epsilon_2 \leq t_2 \leq A_2$,
	$0 \leq u \leq 1$.
	Recalling Remark \ref{remark: hodge star},
	we have that pulled back to $\partial B$,
	\begin{align*}
	b_1 = \dsst \left(Q_1 - \frac{n_1}{2} I \right)  \frac{dt_1}{t_1}  + \ast_1^{-1} \frac{d \ast_1}{du} \, du \text{ and } b_2 = \dsst \left(Q_2 - \frac{n_2}{2} I \right)  \frac{dt_2}{t_2}.
	\end{align*}
	Since $b_2$ commutes with the resolvent,
	the formula
	\eqref{eq: defn of tilde omega MT}
	simplifies greatly,
	and we may compute the contour integral in \eqref{eq: defn of omega MT}
	to obtain that $\omega_{MT}$ pulled back to $\partial B$ is
	\begin{multline*}
	\omega_{MT} = 
	\Tr (-1)^Q e^{-\Delta(t_1, t_2, u)} \frac{Q_1}{t_1} \frac{Q_2}{t_2} \, dt_1 \,dt_2 \\
	+ \Tr (-1)^Q e^{-\Delta(t_1, t_2, u)} \ast_1^{-1} \frac{d \ast_1}{du} \left(Q_2 - \frac{n_2}{2} I \right) \frac{1}{t_2}  \, du\, dt_2,
	\end{multline*}
	where we have used that
	\begin{align*}
	\Tr (-1)^Q Q_1 e^{-\Delta(t_1, t_2, u)} =\Tr (-1)^Q Q_2 e^{-\Delta(t_1, t_2, u)} =\Tr (-1)^Q  e^{-\Delta(t_1, t_2, u)}  =0, 
	\end{align*}
	which follow from arguments similar to an argument in the proof
	of Theorem \ref{theorem: even vanishing of MT}.
	Since $\omega_{MT}$ is closed, we have
	\begin{align*}
	\delta (t_1^{s_1} t_2^{s_2} \omega_{MT}) = s_1 t_1^{s_1 - 1} t_2^{s_2} \Tr (-1)^Q e^{-\Delta(t_1, t_2, u)} \ast_1^{-1} \frac{d \ast_1}{du} \left(Q_2 - \frac{n_2}{2} I \right) \frac{1}{t_2} \, du \, dt_1 \, dt_2.
	\end{align*}
	Assume that $\Real s_1$ and $\Real s_2$ are both large.
	We apply Stokes' theorem on $B$ to obtain
	\begin{equation} \label{stokes}
	\iiint_B \delta^\M_1 (t_1^{s_1} t_2^{s_2} \omega_{MT}) = \iint_{\partial B} t_1^{s_1} t_2^{s_2} \omega_{MT}.
	\end{equation}
	The integral over $B$ is
	\begin{align*}
	& \iiint_B \delta^\M_1 (t_1^{s_1} t_2^{s_2} \omega_{MT}) \\
	& =  s_1 \int_{0}^1  \int_{\epsilon_2}^{A_2} \int_{\epsilon_1}^{A_1}
	t_1^{s_1 - 1} t_2^{s_2} \Tr (-1)^Q e^{-\Delta(t_1, t_2, u)} \ast_1^{-1} \frac{d \ast_1}{du}  \left(Q_2 - \frac{n_2}{2} I \right) \frac{1}{t_2} \,  dt_1 \, dt_2 \, du.
	\end{align*}
	The boundary $\partial B$ is a cube with six faces,
	described by $u=0$, $u=1$, $t_1 = \epsilon_1$, $t_1 = A_1$, $t_2 = \epsilon_2$, and $t_2 = A_2$. Since $\omega_{MT}$ vanishes on the latter two faces,
	the integral over $\partial B$ consists of four terms:
	\begin{align}
	& \int_{\epsilon_2}^{A_2} \int_{\epsilon_1}^{A_1} t_1^{s_1} t_2^{s_2} \Tr (-1)^Q e^{-\Delta(t_1, t_2, 1)} \frac{Q_1}{t_1} \frac{Q_2}{t_2} \, dt_1 \, dt_2 ~~~~ &(u=1) \\
	-& \int_{\epsilon_2}^{A_2} \int_{\epsilon_1}^{A_1} t_1^{s_1} t_2^{s_2} \Tr (-1)^Q e^{-\Delta(t_1, t_2, 0)} \frac{Q_1}{t_1} \frac{Q_2}{t_2} \, dt_1 \, dt_2 ~~~~&(u=0) \\
	\label{3rd} +& \int_{0}^{1} \int_{\epsilon_2}^{A_2} \epsilon_1^{s_1} t_2^{s_2} \Tr (-1)^Q e^{-\Delta(\epsilon_1, t_2, u)} \ast_1^{-1} \frac{d \ast_1}{du}  \left(Q_2 - \frac{n_2}{2} I \right) \frac{1}{t_2} \, dt_2 \, du ~~~~&(t_1 = \epsilon_1)\\
	\label{4th} +& \int_{0}^{1} \int_{\epsilon_2}^{A_2} A_1^{s_1} t_2^{s_2} \Tr (-1)^Q e^{-\Delta(A_1, t_2, u)} \ast_1^{-1} \frac{d \ast_1}{du}  \left(Q_2 - \frac{n_2}{2} I \right) \frac{1}{t_2} \, dt_2 \, du ~~~~&(t_1 = A_1)
	\end{align}
	For fixed $\epsilon_2$ and $A_2$, we will now study the limits $\epsilon_1 \to 0^+$ and $A_1 \to \infty$.
	By the heat kernel estimates,
	the term \eqref{3rd} tends to $0$ as $\epsilon_1 \to 0^+$,
	and the term \eqref{4th} tends to $0$ as $A_1 \to \infty$.
	Thus from \eqref{stokes} we obtain
	\begin{align*}
	& s_1 \int_{0}^1  \int_{\epsilon_2}^{A_2} \int_{0}^\infty t_1^{s_1-1} t_2^{s_2} \Tr (-1)^Q e^{-\Delta(t_1, t_2, u)} \ast_1^{-1} \frac{d \ast_1}{du}  \left(Q_2 - \frac{n_2}{2} I \right) \frac{1}{t_2} \, du \, dt_2 \, dt_1 \\
	=& \int_{\epsilon_2}^{A_2} \int_{0}^{\infty} t_1^{s_1} t_2^{s_2} \Tr (-1)^Q e^{-\Delta(t_1, t_2, 1)} \frac{Q_1}{t_1} \frac{Q_2}{t_2} \, dt_1 \, dt_2 \\
	-& \int_{\epsilon_2}^{A_2} \int_{0}^{\infty} t_1^{s_1} t_2^{s_2} \Tr (-1)^Q e^{-\Delta(t_1, t_2, 0)} \frac{Q_1}{t_1} \frac{Q_2}{t_2} \, dt_1 \, dt_2.
	\end{align*}
	We may now take the limits $\epsilon_2 \to 0^+$ and $A_2 \to \infty$ to obtain
	\begin{align*}
	& s_1 \int_{0}^1  \int_{0}^{\infty} \int_{0}^\infty t_1^{s_1-1} t_2^{s_2} \Tr (-1)^Q e^{-\Delta(t_1, t_2, u)} \ast_1^{-1} \frac{d \ast_1}{du}  \left(Q_2 - \frac{n_2}{2} I \right) \frac{1}{t_2} \, du \, dt_2\, dt_1 \\
	=& \int_{0}^{\infty} \int_{0}^{\infty} t_1^{s_1} t_2^{s_2} \Tr (-1)^Q e^{-\Delta(t_1, t_2, 1)} \frac{Q_1}{t_1} \frac{Q_2}{t_2} \, dt_1 \, dt_2 \\
	-& \int_{0}^{\infty} \int_{0}^{\infty}  t_1^{s_1} t_2^{s_2} \Tr (-1)^Q e^{-\Delta(t_1, t_2, 0)} \frac{Q_1}{t_1} \frac{Q_2}{t_2} \, dt_1\, dt_2.
	\end{align*}
	Multiplying both sides by $\frac{1}{\Gamma(s_1)\Gamma(s_2)}$ gives that
	\begin{align*}
	&s_1 \int_0^1 \zeta\left(s_1, s_2; \Delta(t_1, t_2, u), (-1)^Q \ast_1^{-1} \frac{d \ast_1}{du}  \left(Q_2 - \frac{n_2}{2} I \right) \right) \, du  \\
	&=  \zeta\left(s_1, s_2; \Delta(t_1, t_2, 1), (-1)^Q Q_1 Q_2 \right) - \zeta\left(s_1, s_2;  \Delta(t_1, t_2, 1), (-1)^Q Q_1 Q_2 \right).
	\end{align*}
	We have shown that the equality holds for $\Real s_1$ and $\Real s_2$ large,
	but since both sides possess unique meromorphic continuations to all of $\mathbb{C}^2$,
	in fact the equality holds everywhere.
	In particular, using the Definition \ref{definition: multi-torsion} of multi-torsion, we have the following equality of derivatives at the origin:
	\begin{multline*}
	\left. \frac{\partial^2}{\partial s_1 \partial s_2} \right|_{(s_1, s_2)=(0,0)} s_1 \int_0^1   \zeta\left(s_1, s_2; \Delta(t_1, t_2, u), (-1)^Q \ast_1^{-1} \frac{d \ast_1}{du}  \left(Q_2 - \frac{n_2}{2} I \right) \right)
	\, du  \\
	= MT(1) - MT(0).
	\end{multline*}
	Since the  multi-zeta function in the integral is holomorphic
	at $(0,0)$, the left-hand side is equal to
	\begin{equation*}
	\int_0^1 \left. \frac{\partial}{\partial s_2} \right|_{(s_1, s_2)=(0,0)}  \zeta\left(s_1, s_2; \Delta(t_1, t_2, u), (-1)^Q \ast_1^{-1} \frac{d \ast_1}{du}  \left(Q_2 - \frac{n_2}{2} I \right) \right) \, du.
	\end{equation*}
	Note that $(-1)^Q \ast_1^{-1} \frac{d \ast_1}{du}  \left(Q_2 - \frac{n_2}{2} I \right)$
	is an operator of the form $\alpha_\Gamma$ in the assumptions
	of Corollary \ref{coro: partial derivative vanish},
	and the assumptions  of the present theorem ensure
	that the remaining assumptions of that corollary hold.
	Thus we may apply that corollary
	to obtain that the integrand vanishes.
	This shows $MT(1) - MT(0) = 0$,
	which proves the theorem.
\end{proof}

\section{The eta invariant}
\label{section: eta}

In the series of papers \cites{aps-I, aps-II, aps-III}, Atiyah-Patodi-Singer
introduced the eta invariant
as a measure of ``spectral asymmetry"
and proved an index theorem for manifolds with boundary.
In this section, we provide another example of the utility of 
our differential forms formalism
by observing that Atiyah-Patodi-Singer's
theorem computing the variation of the eta invariant
is a consequence of the closedness of a certain one-form
$\omega_\eta$ on the space of elliptic operators.

Let $M$ be a compact manifold of dimension $n$.
Let $E \to M$ be a vector bundle of rank $k$.
Let $B$ be an elliptic differential operator of order $r >0 $
that is symmetric with respect to some metric on $E$.
Set 
\begin{align}
\eta(s; B) :=  \zeta \left( \frac{s+1}{2}; B^2, B \right).
\end{align}
Recalling the definition  \eqref{equation: heat zeta def} of the zeta function, we have for $\Real s$ large that
\begin{align} \label{equation: rewrite eta}
\eta(s; B) = \frac{1}{\Gamma\left(\frac{s+1}{2} \right)}  \int_0^\infty t^{\frac{s+1}{2}} \Tr B e^{-t B^2} \, \frac{dt}{t}.
\end{align}
By a  theorem of Atiyah-Patodi-Singer \cite{aps-III}, $\eta(s; B)$
is holomorphic at $s=0$, ensuring that the following definition 
of the eta invariant $\eta(B)$ makes sense:
\begin{align} \label{equation: eta invariant}
\eta(B) := \eta(0; B).
\end{align}

For some smooth parameter
space $\U$, let $u\in \U \mapsto B=B(u)$ be a smooth family of elliptic operators of order $r$.
Assume that for every $u \in \U$, 
there exists some metric $h(u)$
on $E$ with respect to which $B(u)$ is symmetric.
We will use $\delta$ to denote the exterior derivative on $\U$.

Consider the following $\mathbb{C}$-valued one-form on $\U$:
\begin{align}
\omega_\eta := \Tr (\delta B)  e^{-B^2}.
\end{align}
To explain the significance of $\omega_\eta$ to the eta invariant,
let $B_1$ be a fixed symmetric elliptic operator.
Let $C_{B_1}$ be the curve in the space of elliptic operators
parametrized by $t \in (0, \infty) \mapsto B(t) := \sqrt{t}B_1$.
Pulled back to $C_{B_1}$, we have
\begin{equation*}
\delta B = \frac{1}{2} t^{-1/2} B_1 \, dt \text{ and }
\omega_\eta = \frac{1}{2} t^{-1/2} \Tr B_1 e^{-tB_1^2} \, dt.
\end{equation*}
Thus for $\Real s$ large,
\begin{align} 
\eta(s; B_1) = \frac{2}{\Gamma\left(\frac{s+1}{2} \right)}  \int_{C_{B_1}} t^{\frac{s}{2}} \omega_\eta,
\end{align}
giving us the interpretation:
\begin{displayquote}
	The eta invariant $\eta(B_1)$ may be interpreted as $\frac{2}{\Gamma\left( 1/2 \right)}$ times the regularized integral of $\omega_\eta$ over the curve
	$C_{B_1}$ in the space of elliptic operators.
\end{displayquote}
We have the fundamental result:
\begin{lemma}
	$\omega_\eta$ is closed on $\U$.
\end{lemma}
\begin{proof}
	As usual, we will work with the resolvent
	$R_z := (z- B^2)^{-1}$ rather than directly with the heat operator $e^{-B^2}$.
	Fix an integer $N$ so that $2rN - r > n$,
	which ensures that $(\delta B) R_z^N$
	is trace-class, so that we may define 
	the one-form $\tilde{\omega}_\eta$ (whose dependence on $z$
	we suppress) by
	\begin{align*}
	\tilde{\omega}_\eta := \Tr \, (\delta B)  R_z^N.
	\end{align*}
	By the Cauchy integral formula, $\omega_\eta$ and $\tilde{\omega}_\eta$ are related by the identity
	\begin{align*}
	\omega_\eta &= \frac{1}{2\pi i} \frac{1}{(N-1)!}  \int_C e^{-z} \, \tilde{\omega}_\eta  \, dz,
	\end{align*}
	where $C$ is an appropriate contour in the complex plane.
	Thus to prove $\omega_\eta$ is closed, it suffices to prove $\tilde{\omega}_\eta$ is closed, which we will now do by a short computation.
	We have
	\begin{align}
	\delta \tilde{\omega}_\eta &= - \sum_{j=1}^N \Tr \, (\delta B) R_z^j [\delta(B^2 )] R_z^{N-j+1}   \nonumber \\
	&= - \sum_{j=1}^N \Tr \, (\delta B) R_z^j [B (\delta B)  + (\delta B) B ] R_z^{N-j+1}.
	\label{sum for eta tilde}
	\end{align}
	We may compute that
	\begin{align}
	\Tr \, (\delta B) R_z^j B( \delta B)  R_z^{N-j+1}
	&= \Tr \, (\delta B)  B R_z^j (\delta B) R_z^{N-j+1} \nonumber \\
	&= - \Tr \, (\delta B) R_z^{N-j+1} (\delta B)  B R_z^j, \nonumber
	\end{align}
	where we have used that $B$ commutes with $R_z$
	and the signed cyclicity of the trace (Lemma \ref{lemma: adjoint}).
	This shows that the sum in \eqref{sum for eta tilde} vanishes,
	completing the proof.
\end{proof}

Atiyah-Patodi-Singer proved the following 
theorem,
a consequence of which is
the homotopy invariance of
the reduced eta invariant,
defined as a difference of eta invariants associated  
to flat coefficient bundles of the same rank \cite{aps-III}.

\begin{theorem}[Atiyah-Patodi-Singer \cite{aps-I}, \cite{aps-III}]
	For a smooth one-parameter family of elliptic operators $B(u)$,
	the derivative of the eta invariant is
	\begin{align*} 
	\frac{d}{du} \eta(B(u)) = \frac{2}{\Gamma \left( \frac{1}{2} \right) } 
	\left[ \Tr \frac{dB}{du} e^{-t B(u)^2} \right]_{t^{-1/2}}.
	\end{align*}
\end{theorem}

\begin{remark}
The theorem may be proven by applying Stokes' theorem to our closed form $\omega_\eta$.
To explain this, suppose $B(u)$ is a smooth family of elliptic operators for $u \in [0,1]$.
For $A > \e > 0$, let $\Sigma = \Sigma_{A, \e}$ denote the surface
in the space of elliptic operators
parametrized by $(u, t) \mapsto t^{1/2} B(u)$, for $u  \in [0,1]$ and $t \in [\e, A]$.
Pulled back to $\Sigma$,
we have that $\delta B = \frac{1}{2}t^{-1/2} B(u) \, dt + t^{1/2} \frac{dB}{du} \, du$
and thus
\begin{align}
	\omega_\eta = \frac{1}{2}t^{-1/2} \Tr B(u) e^{-tB(u)^2} \, dt + t^{1/2} \Tr \frac{dB}{du} e^{-t B(u)^2} \, du.
\end{align} 
Since $\omega_\eta$ is closed, $\delta \left( t^{\frac{s}{2}} \omega_\eta \right)$ pulled back to $\Sigma$ is
\begin{align*}
	\delta \left( t^{\frac{s}{2}} \omega_\eta \right)
	&=  \frac{s}{2} t^{\frac{s+1}{2}}  \,  \Tr \frac{dB}{du}  e^{-t B(u)^2} \, \frac{dt}{t}  \, du .
\end{align*}
We omit the remaining details of the argument since it is very similar to
that in the proof of Theorem \ref{theorem: variation of T}.
\end{remark}

\begin{appendices}

\end{appendices}

 \bibliographystyle{abbrv} 
 \bibliography{bibliography}

\end{document}